\documentclass[a4paper,11pt]{article}	
\usepackage{graphicx}
\usepackage{relsize}
\usepackage{hyperref}
\usepackage{comment}

\medskipamount
\usepackage[usenames,dvipsnames]{color}

\usepackage{amsmath}
\usepackage{amsthm}
\usepackage{amssymb}
\usepackage{amscd}
\usepackage{enumerate}
\usepackage{pstricks}
\usepackage{pst-node}
\usepackage{esint}
\usepackage{mathtools}
\theoremstyle{plain}
\newtheorem{theorem}{Theorem}[section]
\newtheorem{corollary}[theorem]{Corollary}
\newtheorem{lemma}[theorem]{Lemma}
\newtheorem{proposition}[theorem]{Proposition}

\theoremstyle{definition}
\newtheorem{assumption}[theorem]{Assumption}
\newtheorem{definition}[theorem]{Definition}

\theoremstyle{remark}
\newtheorem{remark}{Remark}

\newcommand\R{\mathbb{R}}
\newcommand\M{\mathbb{M}}
\newcommand\N{\mathbb{N}}

\newcommand\Ra{R}

\renewcommand\boldsymbol{}

\newcommand\calW{\mathcal{W}}

\newcommand{\SO}[1]{\operatorname{SO}(#1)}
\newcommand\id{I}
\newcommand\sym{\operatorname{sym}}
\newcommand\skw{\operatorname{skw}}

\newcommand\tr{\operatorname{tr}}
\newcommand\axl{\operatorname{axl}}

\newcommand\ud{\,\mathrm{d}}
 
 \newcommand\dist{\operatorname{dist}}
\newcommand{\e}{\varepsilon}
\newcommand{\eh}{{\varepsilon{\scriptstyle(h)}}}
\newcommand\eps{\varepsilon}

\title{General homogenization of bending-torsion theory for inextensible rods from 3D elasticity}
\date{\today}

\author{Maroje Marohni\'{c} \\
University of Zagreb, Faculty of Natural Sciences and Mathematics,\\ Department of Mathematics,\\ Bijeni\v{c}ka 30, 10000 Zagreb, Croatia\\ Email:maroje.marohnic@math.hr \vspace{1ex}\\
Igor Vel\v{c}i\'{c}\\
University of Zagreb, Faculty of Electrical Engineering and Computing,\\ Unska 3, 10000 Zagreb, Croatia\\
Email: igor.velcic@fer.hr}

\begin{document}

\maketitle

\begin{abstract}
We derive, by means of $\Gamma$-convergence, the equations of homogenized bending rod starting from $3D$ nonlinear elasticity equations.
The main assumption is that the energy behaves like $h^2$ (after dividing by the order $h^2$ of vanishing volume) where $h$ is the thickness of the body.
We do not presuppose any kind of periodicity and work in the general framework. The result shows that, on a subsequence, we always obtain the equations of bending-torsion rod and identifies, in an abstract formulation, the limiting quadratic form connected with that model. This is a generalization from periodic to non-periodic homogenization of bending-torsion rod theory already present in the literature.
\vspace{10pt}


\end{abstract}

\tableofcontents[1]

|

\section{Introduction} \label{sectionprvi}
This paper is about derivation of homogenized bending-torsion theory for rods, starting from $3D$ elasticity by means of $\Gamma$-convergence.  The main novelty is that we do not presuppose any kind of periodicity, but work in a general framework.
There is a vast literature on deriving  rod, plate and shell equations from $3D$ elasticity.
The first  work in deriving the lower dimensional models by $\Gamma$-convergence techniques was \cite{AcBuPe} where the authors derived the string model. It was well known that the obtained models depend on the assumption what is the relation of the external loads (i.e. the energy) with respect to the thickness of the body $h$.
The first rigorous derivation of higher ordered models was done in \cite{FJM-02,FJM-06}) for the case of bending and von K\'arm\'an plate. The key mathematical ingredient in these cases was the theorem on geometric rigidity.

After these pioneering works there  is a vast literature on the rigourous derivation of lower dimensional models from $3D$ elasticity by means of $\Gamma$-convergence.  We mention only those works that refer to the derivation of the rod theories.

In \cite{MoMu03} the authors derive the bending-torsion rod theory assuming the fixed stored energy density function (without possible oscillations in the material). As usual in bending theories, they assume that the energy is of the order $h^2$, where $h$ is the thickness of the body (after division with the order of vanishing volume, which is $h^2$). In \cite{MoMu04} the authors derive the model in the so called von-K\'arm\'an regime where the order of energy is $h^4$. In \cite{MoraMuller08} the authors analyze the stationary points (i.e. the equations) in the case of bending rod and show that the limit equation is the one corresponding to the limit energy obtained by $\Gamma$-convergence. However, due to nonlinearity, it is not clear that the global minimizers satisfy these equations from which  the authors start the derivation (see \cite{MoraScardia12,DavoliMora1} for details).

 It is important to notice that the bending theory is still large deformation theory (although small strain theory), while the von-K\'arm\'an theory is a small displacement theory where the limit deformation is rigid motion and the energy depends on the correctors. Thus, we can say that the bending theory carries more nonlinearity. We refer the reader to \cite{Scardia09} where the author gave the full asymptotic (higher ordered) theory for curved rods.

This paper deals with the effects of simultaneous homogenization and dimensional reduction.
There is a vast literature on the effects of simultaneous homogenization and dimensional reduction on the limit equations, in different context.
In \cite{gustafsson06} the authors study the effects of simultaneous homogenization and dimensional reduction for linear elasticity system without periodicity assumption introducing the variant of $H$-convergence adapted to dimensional reduction. In \cite{Braides-Fonseca-Francfort-00} the authors study the same effects for nonlinear systems (membrane plate) by means of $\Gamma$-convergence, also without periodicity assumptions. In \cite{Courilleau04} the author studies nonlinear monotone operators in the context of simultaneous homogenization and dimensional reduction, without periodicity assumption. Much earlier in \cite{JurakTutek} the authors study the same effects for the linear rod case where it was assumed that the rod is homogeneous along its central line, but the microstructure is given in the cross section. We also mention the work of Arrieta on Laplace equation and thin domain with an oscillatory boundary (see e.g. \cite{Arrieta96}). Finally we emphasize the works  \cite{Neukamm-10,Neukamm-11}, where the author gave the systematic approach and combined the techniques from \cite{FJM-02,FJM-06} and two scale convergence to obtain the model of homogenized bending rod.

Recently, the techniques from \cite{FJM-02,FJM-06} were combined together with two-scale convergence to obtain the models of homogenized von  K\'arm\'an plate (see \cite{Velcic-12,NeuVel-12}), homogenized von K\'arm\'an shell (see \cite{Hornungvel12}) and homogenized bending plate (see \cite{Horneuvel12,Vel13}). These models were derived under the assumption of periodic oscillations of the material where it was assumed that the material oscillates on the scale $\eps(h)$, while the thickness of the body is $h$. The obtained models depend on the parameter $\gamma=\lim_{h \to 0} \tfrac{h}{\eh}$. In the case of von K\'arm\'an plate the situation $\gamma=0$ corresponds to the case when dimensional reduction dominates and the obtained model is the model of homogenized von K\'arm\'an plate and can be obtained as the limit case when $\gamma \to 0$. Analogously, the situation when $\gamma=\infty$ corresponds to the case when homogenization dominates and can again be obtained as the limit when $\gamma  \to \infty$; this is the model of von K\'arm\'an plate obtained starting from homogenized energy. In the case of von K\'arm\'an shell and bending plate the situation $\gamma=0$ was more subtle and leaded that the models depend on the further assumption of the relation between $\eps(h)$ and $h$. We obtained different models for the case $\eh^2 \ll h \ll \eh$ and $h \sim \eh^2$.

This paper derives the model of bending-torsion rod by simultaneous homogenization and dimensional reduction without any periodicity assumption and  generalize the work \cite{Neukamm-11}. In that work the author derived the bending-torsion rod theory by assuming periodic oscillations of the material on the central line of the rod.
The author used the tool of two-scale convergence, appropriate for periodic homogenization.
Here we show sort of stability result: one obtains the same type of equations starting from any kind of oscillating or non-oscillating material, where the oscillations can be done in any direction (even in the
cross-section). This is, because of these reasons, significant improvement of \cite{Neukamm-11} and uses the general approach from $\Gamma$-convergence. The similar work in this direction was \cite{Velcic-13} where the author derived the model of von-K\'arm\'an plate by means of simultaneous homogenization and dimensional reduction without periodicity assumption and thus generalized earlier work \cite{NeuVel-12}.

This paper, together with \cite{Velcic-13}, is the first treatment of simultaneous homogenization and dimensional reduction without periodicity assumption by variational techniques in the context of higher order models in elasticity, at least to our knowledge (membrane case  was already analyzed in \cite{Braides-Fonseca-Francfort-00}).
The main results are given in Theorem \ref{theorem1} and  Theorem \ref{theorem2}
where the lower bound and the upper bound is proved, respectively. We prove that, on a subsequence,
 the limit energy density is a quadratic form in the strain of the limit deformation (the limit deformation and the strain itself is the standard one for the bending rod case).

\subsection{Notation}
By $B(x,r)$ we denote the ball of radius $r>0$ around $x \in \R^n$ in Euclidean norm. We denote by $e_1,e_2,e_3$ the canonical basis in $\R^3$ and by $\nabla_h$ we denote the operator $\nabla_h=(\partial_1, \tfrac{1}{h} \partial_2, \tfrac{1}{h} \partial_3)$. By $\mathbb{M}^{m \times n}$ we denote the space of matrices with $m$ rows and $n$ columns, by  $\mathbb{M}^n$ we denote the space of quadratic matrices of order $n$.  $\mathbb{M}^n_{\sym}$ denotes the space of symmetric matrices of order $n$, while  $\mathbb{M}^n_{\skw}$  denotes the space of skew symmetric matrices of order $n$. For $A \in \mathbb{M}^n$ by $\sym A$ we denote the symmetric part of $A$; $\sym A=\tfrac{1}{2} (A+A^t)$, while by $\skw A$ we denote the skew symmetric part of $A$; $\skw A=\tfrac{1}{2}(A-A^t)$. For $A, B \in \mathbb{M}^n$ by $A \cdot B$ we  denote the  $\tr (AB^t)$.
$\iota:\mathbb{\R}^{3} \to \mathbb{M}^3$ denotes the natural inclusion
$$ \iota(m)=\sum_{i=1}^3 m_{i} e_i \otimes e_1. $$
For $A \in \mathbb{M}^3_{\skw}$  $\textrm{axl} A$  stands for the axial vector of $A$, i.e.,
$Ax= \axl A \wedge x$, for all $x \in \R^3$. It is easy to see that $\axl A=(A_{32}, A_{13},A_{21})^t$.

If $O \subset \R^n$ open,
by $W^{1,p}(O;M)$ we denote the subset of Sobolev space of functions taking values in $M \subset \R^m$ for a.e. $x \in O$. It is easy to see if $M$ is a subspace of $\R^m$ then $W^{1,p}(O;M)$ is a subspace of $W^{1,p}(O;\R^m)$. If $M$ is closed subset of $\R^m$ then $W^{1,p}(O;M)$ is a closed subset of $W^{1,p}(O;\R^m)$
in weak and strong topology.
%$f \in C^0(O)$ by $\|f\|_{C^0}$ we denote the norm $\|f\|_{C^0}=\sup_{x \in %O}|f(x)|$. If $f \in C^1(O)$ we define
%$$\|f\|_{C^1}=\sup_{x \in O}\left(|f(x)|+|\nabla f(x)|\right).$$
%In the analogous way we $\|f\|_{C^0}$, $\|f\|_{C^1}$ if $f$ is vector valued.
For $S \subset \R^n$, by $\chi_S$ we denote the characteristic function of $S$; $\chi_S: \R^n \to \{0,1\}$. By $|S|$ we denote the Lebesgue measure of $S$. For $x \in \R$, by $\lfloor x \rfloor$ we denote the greatest integer less or equal to $x$.

\section{Derivation of the model} \label{sectiondrugi}

Let $\omega \subset \R^2$ be an open connected set with Lipschitz boundary. We define by $\Omega^h =[0,L] \times h\omega$ the reference configuration of the rod-like body. When $h=1$ we omit the superscript and write $\Omega=\Omega^1$. We may assume that the coordinate axes are chosen such that
 \begin{equation} \label{normalizacija}
 \int_{\omega} x_2 \ud x_2 \ud x_3= \int_{\omega} x_3 \ud x_2 \ud x_3= \int_{\omega} x_2 x_3 \ud x_2 \ud x_3=0.
 \end{equation}
We denote the moments of inertia by $\mu_{i}=\int_{\omega} x_i^2 \ud x_2 \ud x_3$ for $i=2,3$ and define $d_{\omega}=(0,x_2,x_3)^t$.

\subsection{General framework}
The following two definitions will give conditions on the energy densities.
\begin{definition}[Nonlinear material law]\label{def:materialclass}
  Let $0<\eta_1\leq\eta_2$ and $\rho>0$. The class
  $\calW(\eta_1,\eta_2,\rho)$ consists of all measurable functions
  $W\,:\,\R^{3 \times 3}\to[0,+\infty]$ that satisfy the following properties:
  \begin{align}
    \tag{W1}\label{ass:frame-indifference}
    &W\text{ is frame indifferent, i.e.}\\
    &\notag\qquad W(\Ra\boldsymbol F)=W(\boldsymbol F)\quad\text{ for
      all $\boldsymbol F\in\M^3$, $\boldsymbol R\in\SO
      3$;}\\
    \tag{W2}\label{ass:non-degenerate}
    &W\text{ is non degenerate, i.e.}\\
    &\notag\qquad W(\boldsymbol F)\geq \eta_1 \dist^2(\boldsymbol F,\SO 3)\quad\text{ for all
      $\boldsymbol F\in\M^3$;}\\
    &\notag\qquad W(\boldsymbol F)\leq \eta_2\dist^2(\boldsymbol F,\SO 3)\quad\text{ for all
      $\boldsymbol F\in\M^3$ with $\dist^2(\boldsymbol F,\SO 3)\leq\rho$;}\\
    \tag{W3}\label{ass:stressfree}
    &W\text{ is minimal at $\id$, i.e.}\\
    &\notag\qquad W(\id)=0;\\
    \tag{W4}\label{ass:expansion}
    &W\text{ admits a quadratic expansion at $\id$, i.e.}\\
    &\notag\qquad W(\id+\boldsymbol G)=Q(\boldsymbol G)+o(|\boldsymbol G|^2),\qquad\text{for all }\boldsymbol G\in\M^3,\\
    &\notag\text{where $Q\,:\,\M^3\to\R$ is a quadratic form.}
 \end{align}
\end{definition}
In the following definition we state our assumptions on the family $(W^h)_{h>0}$.
\begin{definition}[Admissible composite material]\label{def:composite}
  Let $0<\eta_1\leq\eta_2$ and $\rho>0$. We say that a family $(W^h)_{h>0}$
  \begin{equation*}
    W^h:\Omega\times \mathbb{M}^{3} \to \R^+\cup\{+\infty\},
  \end{equation*}
  describes an admissible composite material of class $\mathcal W(\eta_1,\eta_2,\rho)$ if
  \begin{enumerate}[(i)]
  \item  For each $h>0$, $W^h$ is almost surely equal to a Borel function on $\Omega\times\R^{3\times 3}$,
  \item $W^h(x,\cdot)\in\mathcal W(\eta_1,\eta_2,\rho)$ for  every $h>0$ and almost every $x\in\Omega$.
  \item there exists a monotone function $r:\R^+\to\R^+\cup\{+\infty\}$, such that $r(\delta)\to 0$ as
  $\delta\to 0$ and
  \begin{equation}\label{eq:94}
    \forall \boldsymbol G\in\R^{3\times 3}\,:\,\limsup_{h \to 0} |W^h(x,\id+\boldsymbol G)-Q^h(x,\boldsymbol G)|\leq r(|\boldsymbol G|) |\boldsymbol G|^2,
  \end{equation}
  for  almost all $x\in\Omega$, where $Q^h(x,\cdot)$ is a quadratic form given in Definition \ref{def:materialclass}.
\end{enumerate}
\end{definition}
Notice that  for each $h>0$ $Q^h$  can be written as the pointwise limit
\begin{equation} \label{defQ} (x,G) \to Q^h(x,G) := \lim_{\e \to 0}
\tfrac{1}{\e^2} W^h(x, Id+\e G).
\end{equation}
Therefore, it  inherits the measurability properties of $W^h$.

\begin{lemma}
  \label{lem:111}
  Let $(W^h)_{h>0}$ be as in Definition~\ref{def:composite} and let $(Q^h)_{h>0}$ be the family of the quadratic forms
  associated to $(W^h)_{h>0}$ through the expansion \eqref{ass:expansion}. Then
  \begin{enumerate}[(i)]
  \item[(Q1)] for all $h>0$ and almost all $x\in\Omega$ the map $Q^h(x,\cdot)$ is quadratic and
      satisfies
    \begin{equation*}
      \eta_1|\sym \boldsymbol G|^2\leq Q^h (x,\boldsymbol G)=Q^h(x,\sym \boldsymbol G)\leq \eta_2|\sym \boldsymbol G|^2\qquad\text{for all $ \boldsymbol G\in\M^3$.}
    \end{equation*}
  \end{enumerate}
\end{lemma}
\begin{proof}
(Q1) follows from (W2).
\end{proof}
\begin{remark}\label{ocjenkv}
From (Q1)  we also obtain that
\begin{eqnarray} \label{ocjenakv1} && \\ \nonumber
|Q^h(x,G_1)-Q^h(x,G_2)| &\leq& \eta_2 |\sym G_1-\sym G_2|\cdot |\sym G_1+\sym G_2|,\\ & & \nonumber \hspace{10ex} \textrm{for a.e. } x \in \Omega,\forall h>0, G_1,G_2 \in \R^{3 \times 3}.
\end{eqnarray}
\end{remark}
\begin{comment}
Let $f^h \in L^2(\Omega^h;\R^3)$ be an external body force applied to the rod.
Given deformation $v \in W^{1,2}(\Omega^h;\R^3)$ the total energy per unit cross-section associated to $v$ is given by
$$ \mathcal{F}^h(v)=\tfrac{1}{h^2} \int_{\Omega^h} W^h(x,\nabla v)\ud x-\tfrac{1}{h^2}\int_{\Omega^h} f^h \cdot v \ud x. $$
We rescale $\Omega^h$ to the domain $\Omega=[0,L] \times \omega$ and we rescale deformation according to this change of variable
$$ y^h(x_1,x_2,x_3)=v(x_1,hx_2,hx_3). $$
We also assume, for simplicity that $f^h(x)=f^h(x_1)$. The energy functional can be written in the following form \begin{equation}
\mathcal{F}^h(v)=\mathcal{I}^h(y)= \int_{\Omega} W(\nabla_h y) \ud x-\int_{\Omega} f^h \cdot y \ud x.
\end{equation}
\end{comment}
We will assume that we are in bending regime, i.e., that the energy of minimizing sequence behaves like this
\begin{equation}\label{uvjetnastationary}
\int_{\Omega} W^h(x, \nabla_h y^h) \ud x \leq Ch^2, \textrm{ for some } C>0.
\end{equation}
This assumption can be replaced {\color{Blue} by} the assumption on the scaling of external loads, see \cite{FJM-06} for details.
\begin{theorem}
Assume that the sequence of deformations $y^h \in W^{1,2}(\Omega;\R^3)$  satisfies
 $$ \int_{\Omega} \dist^2(\nabla_h y^h, \SO 3) \ud x \leq C_1h^2,$$
 for some $C_1>0$, independent of $h$.
 Then there exist a constant $C>0$ and a sequence $(R^h)_{h>0} \subset C^{\infty}([0,L];\R^{3 \times 3})$, such that $R^h(x_1) \in \SO 3$ for every $x_1 \in [0,L]$ and
\begin{equation}\label{odnos1}
\|\nabla_h y^h -R^h\|_{L^2} \leq Ch,
\end{equation}
\begin{equation}\label{odnos2}
\|(R^h)'\|_{L^2}+\|h (R^h)''\|_{L^2} \leq C.
\end{equation}
\end{theorem}
\begin{proof}
See the proof of  \cite[Proposition 4.1]{MoraMuller08}.
\end{proof}
From the expression \eqref{odnos2} we conclude that the sequence $(R^h)_{h>0}$, on a subsequence, converges weakly in $W^{1,2}([0,L];\R^{3 \times 3})$ (and thus strongly in $L^\infty([0,L];\R^{3 \times 3})$).
\subsection{Characterization of relaxation field}
We will need the following characterization of the rod deformation.
\begin{lemma}\label{lmgriso}
 There exists a constant $C(\omega)>0$ independent of $h$ such that for given $u \in W^{1,q}(\Omega,\R^3)$, $1<q<\infty$, we have that
\begin{equation}\label{dekompozicija2}
u=a^h+B^h(x_1,hx_2,hx_3)^t+\left(\begin{array}{c}-(\varphi^h_{1})'(x_1)x_2-(\varphi^h_{2})'(x_1)x_3+z^h_{1}(x)\\
\tfrac{1}{h} \varphi^h_{1}(x_1) +w^h(x_1) x_3+z^h_{2}(x)\\
\tfrac{1}{h} \varphi^h_{2}(x_1)-w^h(x_1)x_2+z^h_{3}(x)
\end{array} \right),
\end{equation}
where $a_h \in \R^3$, $B^h \in \M^3_{\textrm{skw}}$, $\varphi_\alpha \in W^{2,q}([0,L])$ for $\alpha=1,2$, $w^h \in W^{1,q}([0,L])$, $z^h \in W^{1,q}(\Omega;\R^3)$ and
\begin{eqnarray}\label{ocjena1}
&&\| \varphi^h_{1}\|_{W^{2,q}}+\| \varphi^h_{2}\|_{W^{2,q}}+\| w^h\|_{W^{1,q}} \leq C(\omega)\|\sym \nabla_h u\|_{L^q}, \\
\label{ocjena2}
&&\int_{\Omega} |z^h|^q+\int_{\Omega} |\nabla_h z^h|^q \leq C(\omega) \|\sym \nabla_h u\|_{L^q}.
\end{eqnarray}
\end{lemma}
\begin{proof}
The proof follows from the Griso's decomposition (see \cite[Theorem 2.1]{Grisodec08}): there are functions $U_e$ and $\bar{u}$ such that $u=U_e+\bar{u}$, where $\bar{u} \in W^{1,q}(\Omega;\R^3)$ and $U_e$ is the  elementary displacement, i.e.,  there are functions $\mathcal{U}  \in W^{1,q}([0,L];\R^3)$ and $\mathcal{R} \in W^{1,q}([0,L];\R^3)$ which are independent of $x_2$ and $x_3$ variables such that
\begin{equation*}
U_e=\mathcal{U} + \mathcal{R} \times (h x_2e_2 + h x_3e_3).
\end{equation*}
Also the following estimates hold
\begin{align}
&\left\Vert \bar{u}\right\Vert_{L^q}\leq C(\omega) h \left\Vert \sym \nabla_h u\right\Vert_{L^q}   \quad \left\Vert \nabla_h \bar{ u}\right\Vert_{L^q} \leq C(\omega) \left\Vert \sym \nabla_h u\right\Vert_{L^q}, \label{ocjenagr1}\\
&h\left\Vert \mathcal{R}'  \right\Vert_{L^q}+\left\Vert\mathcal{U}_1' \right\Vert_{L^q}+\left\Vert \mathcal{U}_2' -\mathcal{R}_3 \right\Vert_{L^q}+\left\Vert \mathcal{U}_3' +\mathcal{R}_2 \right\Vert_{L^q} \leq C(\omega) \left\Vert \sym \nabla_h u\right\Vert_{L^q}. \label{ocjenagr2}
\end{align}
We define the functions:
\begin{align*}
a^h &= \left(\mathcal{U}_1(0), \mathcal{U}_2(0),\mathcal{U}_3(0)\right),\\
B^h &=\left( \begin{array}{ccc}
0 & -\mathcal{R}_3(0) & \mathcal{R}_2(0) \\
\mathcal{R}_3(0) & 0 & - \mathcal{R}_1(0) \\
-\mathcal{R}_2(0) & \mathcal{R}_1(0) & 0
\end{array} \right),
\\
\varphi^h_1(x_1) &=h\left(\int_0^{x_1} \mathcal{R}_3(t) dt  - x_1 \mathcal{R}_3(0)\right),\\
\varphi^h_2(x_1) &=h\left(- \int_0^{x_1} \mathcal{R}_2(t) dt + x_1 \mathcal{R}_2(0)\right),\\
 w^h(x_1) &=-h ({\mathcal{R}}_1(x_1) -\mathcal{R}_1(0)), \\
 z^h_1(x) &= \mathcal{U}_1(x_1) -\mathcal{U}_1(0) + \bar{u}_1(x),\\
 z^h_2(x) &= \mathcal{U}_2 (x_1) - \mathcal{U}_2(0)- \int_0^{x_1} \mathcal{R}_3(t) dt +\bar{u}_2(x),\\
 z^h_3(x) &= \mathcal{U}_3(x_1)  -\mathcal{U}_3(0)+ \int_0^{x_1} \mathcal{R}_2(t) dt +\bar{u}_3(x).
 \end{align*}
It is straightforward to check that  (\ref{dekompozicija2}) holds. To prove the estimates (\ref{ocjena1}) and (\ref{ocjena2}) we use the Poincar\'{e} inequality, (\ref{ocjenagr1}) and (\ref{ocjenagr2}) to deduce
\begin{eqnarray*}
\| \varphi^h_{1}\|_{W^{2,q}}+\| \varphi^h_{2}\|_{W^{2,q}}+\| w^h\|_{W^{1,q}} \leq  C_P h \Vert \mathcal{R}'\Vert_{L^q} \leq C(\omega)\|\sym \nabla_h u\|_{L^q},
\end{eqnarray*}
and
\begin{eqnarray*}
\Vert z^h\Vert_{L^q} +\Vert \nabla_h z^h \Vert_{L^q}
&\leq & C_P\left( \left\Vert \mathcal{U}_1'\right\Vert_{L^q}+ \left\Vert \mathcal{U}_2' -\mathcal{R}_3 \right\Vert_{L^q}+\left\Vert \mathcal{U}_3' +\mathcal{R}_2 \right\Vert_{L^q}\right)  \\ & &
+\Vert  \bar{u}\Vert_{L^{q}}+\Vert\nabla_h  \bar{u}\Vert_{L^{q}}  +h \Vert \mathcal{R}'\Vert_{L^q}
\leq C(\omega) \left\Vert \sym \nabla_h u\right\Vert_{L^q} .
\end{eqnarray*}
\end{proof}
\begin{corollary} \label{cormaroje}
Let  the sequence $(u^h)_{h>0} \subset  W^{1,q}(\Omega,\R^3)$ is such that  $(\|\sym \nabla_h u^h\|_{L^q})_{h>0}$ is bounded, $(u^h_1,hu^h_{2},hu^h_{3})$ converges to zero strongly in $L^q$ and $\int_{\omega} x_3 u^h_2$ (or $\int_{\omega} x_2 u^h_3$) converges to zero strongly in $L^q$. Then there are sequences $(a^h)_{h>0} \in \R^3$, $(B^h)_{h>0} \in \mathbb{M}^3_{\skw}$, $(z^h)_{h>0} \subset W^{1,q}(\Omega;\R^3)$, $(\varphi^h_{\alpha})_{h>0} \subset W^{2,q}([0,L];\R^3)$, for $\alpha=1,2$, and $(w^h)_{h>0} \subset W^{1,q}([0,L];\R^3)$ such that \eqref{dekompozicija2}, \eqref{ocjena1}, \eqref{ocjena2} is valid. Moreover, we have:
\begin{enumerate}
\item
$(a^h_1,ha^h_2,ha^h_3) \to 0$, $hB^h \to 0$, $z^h_{1} \to 0$ strongly in $L^q$, $w^h \to 0$ strongly in $L^q$ and  for $\alpha=1,2$ $\varphi^h_{\alpha} \to 0$  strongly in $W^{1,q}$ as $h \to 0$.
\item  For the following  decomposition of $z^h$,
$z^h=\overline{z}^h+\tilde{z}^h$, where $\overline{z}^h=\int_{\omega} z^h$, we have that
$\bar{z}^h_1 \to 0$ strongly in $L^q$ and $\|\tilde{z}^h\|_{L^{q}} \leq C(\omega)h\|\sym \nabla_h u^h\|_{L^q}$, for some $C(\omega)>0$.
\item  There are sequences $(A^h)_{h>0} \subset  W^{1,q}([0,L];\mathbb{M}^3_{\skw})$ and  $(v^h)_{h>0} \subset W^{1,q}(\Omega;\R^3)$, such that , $A^h \to 0$ and $v^h \to 0$ strongly in $L^{q}$ and the following decomposition holds
$$
\sym \nabla_h u^h=\sym \iota\left((A^h)' d_{\omega}\right)+\sym \nabla_h z^h=\sym \iota \left((A^h)' d_{\omega}\right)+\sym \nabla_h v^h+O(h),
$$
where $O(h) \to 0$ strongly in $L^q$ as $h \to 0$. Moreover, we obtain that
\begin{equation}\label{ocjenaahvh}
\|A^h\|_{W^{1,q}}+\|v^h \|_{L^{q}}+\|\nabla_h v^h \|_{L^{q}}\leq C(\omega) \|\sym \nabla_h u^h \|_{L^q}.
\end{equation}
\end{enumerate}
\end{corollary}
\begin{proof}
Since $\int_{\omega}\tilde{z}^h=0$ we conclude from the Poincar\'{e} inequality that
$$
\|\tilde{z}^h\|_{L^q} \leq C(\omega) h\|\nabla_h z^h\|_{L^q} \leq C(\omega)h \|\sym \nabla_h u^h\|_{L^q}.
$$
Thus, $\tilde{z}^h \to 0$ strongly in $L^q$.
After redefining $a^h$ and $B^h$ we can assume that
\begin{equation} \label{norma22}
\int_{\Omega} z^h=\int_0^L \overline{z}^h=\int_0^L w^h=\int_0^L \varphi^h_{\alpha}=\int_0^L x_1\varphi^h_{\alpha}=0, \textrm{ for } \alpha=1,2.
\end{equation}
Integrating  the first equation in \eqref{dekompozicija2}  over $\omega$ and taking into account the choice of coordinate axes (\ref{normalizacija}), we conclude that $a^h_1 |\omega|+\overline{z}^h_{1} \to 0$ strongly in $L^q(\omega)$. From this, by integration over $[0,L]$, we obtain that $a^h_1 \to 0$ in $L^q(\Omega)$ and, consequently,  $\overline{z}^h_{1} \to 0$ strongly in $L^q(\Omega)$. By multiplying the \eqref{dekompozicija2} with $x_2$ and $x_3$ and taking into account (\ref{ocjena1}) and (\ref{ocjena2}) we obtain that $hB^h_{12}$ and $h B^h_{13}$ are bounded in $L^q$ norm.

We multiply the second and third equation of \eqref{dekompozicija2} by
$h(x_1-\tfrac{L}{2})$, integrate over $\Omega$ and take the limit as $h\to 0$ to obtain that $hB^h_{12} \to 0$ and $h B^h_{13} \to 0$ strongly in $L^q$.  Again, integrating the second and third equation of \eqref{dekompozicija2} over $\Omega$ and taking the limit as $h \to 0$  we deduce  that $ha^h_2 \to 0$ and  $ha^h_3 \to 0$ in $L^q$.  We also obtain that $\varphi^h_{\alpha} \to 0$ strongly in $W^{1,q}$, since it is bounded in $W^{2,q}$.

We multiply the second equation in \eqref{dekompozicija2}  by $x_3$ and integrate over $\omega$. Using the decomposition of $z^h$ we conclude that $hB^h_{23}+w^h \to 0$ strongly in $L^q$. From this, using \eqref{norma22}, it follows that $hB^h_{23} \to 0$ and $w^h \to 0$ strongly in $L^q$. This finishes the proof of (a) and (b).

To  prove (c) we take the sequence $(p^h_2,p^h_3)_{h>0} \subset C^\infty((0,L),\R^2)$ such that
$$
\|p^h-(\overline{z}^h_2,\overline{z}^h_3)\|_{L^q} \to 0,\quad \|p^h\|_{W^{1,q}} \leq C\|(\overline{z}^h_2,\overline{z}^h_3)\|_{W^{1,q}},\quad h\|p^h\|_{W^{2,q}} \to 0.
$$
for some $C>0$. The sequence $p^h$ can be constructed by mollification of $(\overline{z}^h_2,\overline{z}^h_3)$ such that the  mollifiers are of radius $r_h \gg h$.
We define
$$ v^h=z^h-(0,p^h_2,p^h_3)^t+(hx_2 p^h_2+hx_3 p^h_3,0,0)^t,\quad O(h)=(-hx_2  (p^h_2)'-hx_3 (p^h_3)') e_1 \otimes e_1. $$
and conclude the proof.

 From the first equation in \eqref{dekompozicija2} (after integration over $\omega$) we obtain that $a^h_1 |\omega|+\overline{z}^h_{1} \to 0$ strongly in $L^q$. From this, by integration over $[0,L]$, we obtain $a^h_1 \to 0$ and then $\overline{z}^h_{1} \to 0$ strongly in $L^q$.
    We also conclude that $hB^h_{12}$, $h B^h_{13}$ is bounded.
   From the second and third equation of \eqref{dekompozicija2} (after multiplication with $h$) we then obtain that $hB^h_{12} \to 0$, $h B^h_{13} \to 0$. This is done by multiplication with $x_1-\tfrac{L}{2}$ and then integrating over $\Omega$. Now only integration over $\Omega$ also gives $ha^h_2 \to 0$, $ha^h_3 \to 0$.
   We also obtain that
   $\varphi^h_{\alpha} \to 0$ strongly in $W^{1,q}$, since it is bounded in $W^{2,q}$ and converges to zero strongly in $L^q$.

The second equation we multiply with $x_3$ and integrate over $\omega$. Using the decomposition of $z^h$ we conclude that $hB^h_{23}+w^h \to 0$ strongly in $L^q$. From this, using \eqref{norma22}, it follows that $hB^h_{23} \to 0$ and $w^h \to 0$ strongly in $L^q$. This finishes the proof of (a) and (b). To obtain (c) we find the sequence $(p^h_2,p^h_3)_{h>0} \subset C^\infty((0,L),\R^2)$ such that for some $C>0$
$$
\|p^h-(\overline{z}^h_2,\overline{z}^h_3)\|_{L^q} \to 0,\quad \|p^h\|_{W^{1,q}} \leq C\|(\overline{z}^h_2,\overline{z}^h_3)\|_{W^{1,q}},\quad h\|p^h\|_{W^{2,q}} \to 0.
$$
This can be done by mollification of $(\overline{z}^h_2,\overline{z}^h_3)$ with mollifiers of radius $r_h \gg h$.
We define
$$ v^h=z^h-(0,p^h_2,p^h_3)^t+(hx_2 p^h_2+hx_3 p^h_3,0,0)^t,\quad O(h)=(-hx_2  (p^h_2)'-hx_3 (p^h_3)') e_1 \otimes e_1. $$
\end{proof}
\begin{lemma} \label{peter1}
Let $q\geq1$ and let
$A \in  W^{1,q}(\Omega;\mathbb{M}^3_{\skw})$ and $v \in W^{1,q}(\Omega;\R^3)$. Then there exists $u^h \in W^{1,q}(\Omega;\R^3)$ such that
$$\sym \nabla_h u^h=\iota(A' d_{\omega})+\sym \nabla_h v. $$
If $A=0$ and $v=0$ in the neighbourhood of $\{0,L\} \times \omega$ then $u^h=0$ in a neighbourhood of $\{0\} \times \omega$ and $u^h$ is constant in a neighbourhood of $\{L\} \times \omega$.
If
$(A^{h})_{h>0} \subset W^{1,q}([0,L];\mathbb{M}^3_{\skw})$ and $(v^h)_{h>0} \subset W^{1,q}(\Omega;\R^3)$ are such that
 $A^h \to 0$ and $v^h \to 0$ strongly in $L^q$ then $(u_1^h, hu_2^h,hu_3^h) \to 0$ and $\int_{\omega} x_3 u^h_2 \to 0$ and $\int_{\omega} x_2 u^h_3 \to 0$ strongly in $L^q$.
\end{lemma}
\begin{proof}
Everything follows from the definition $u=\Big(A_{12}(x_1)   x_2+A_{13}(x_1) x_3 ,$ $\tfrac{1}{h} \int_0^{x_1} A_{21}(t) \ud t+A_{23}(x_1)x_3 ,$ $ \tfrac{1}{h} \int_0^{x_1} A_{31}(t) \ud t+A_{32}(x_1)x_2 \Big)^t+v$.
\end{proof}
\subsection{Definition of limit energy density}
 We now proceed as in \cite{Velcic-13}
For any open set $O \subset [0,L]$, function $m$  in $L^2(\Omega;\R^3)$ and sequence $(h_n)_{n \in \N}$   monotonly decreasing to zero we define
\begin{eqnarray}
\label{gli} & &\\ \nonumber K^{-}_{(h_n)_{n \in \N}}(m,O)&=&\inf \Big\{\liminf_{n \to \infty}  \int_{O \times \omega} Q^{h_n}\left(x,\iota(m)+\nabla_{h_n} \psi^{h_n}\right) \ud x: \\ \nonumber & & \hspace{-17ex} (\psi_1^{h_n},h_n\psi_2^{h_n},h_n\psi_3^{h_n}) \to 0 \textrm{ strongly in } L^2(O \times \omega;\R^3), \ \int_{\omega} x_3 \psi^{h_n}_2 \to 0 \textrm{ strongly in } L^2(O) \Big\},
%\\ %\nonumber
%&=& \sup_{\mathcal{U} \subset \mathcal{N}(0)} \liminf_{n \to \infty}\inf_{\psi \in H^1(O \times I,\R^3) \atop (\psi_1,h_n\psi_2,h_n\psi_3) \in \mathcal{U}}
 %\int_{O \times \omega} Q^{h_n}\left(x,\iota(m)+\nabla_{h_n} \psi\right) \, dx,
 \\
\label{gls} & &\\ \nonumber K^{+}_{(h_n)_{n \in \N}}(m,O)&=&\inf \Big\{\limsup_{n \to \infty} \int_{A \times I} Q^{h_n}\left(x,\iota(m)+\nabla_{h_n} \psi^{h_n}\right) \ud x:\\ \nonumber & & \hspace{-17ex}  (\psi_1^{h_n},h_n\psi_2^{h_n},h_n \psi_3^{h_n}) \to 0 \textrm{ strongly in } L^2(O \times \omega;\R^3), \ \int_{\omega} x_3 \psi^{h_n}_2 \to 0 \textrm{ strongly in } L^2(O)  \Big\}.
%\\ \nonumber
%&=& \sup_{\mathcal{U} \subset \mathcal{N}(0)} \limsup_{n \to \infty} \inf_{\psi \in H^1(O \times I,\R^3) \atop %(\psi_1,h_n\psi_2,h_n \psi_3) \in  \mathcal{U}} \int_{O \times I} Q^{h_n}\left(x,\iota(m)+\nabla_{h_n} %\psi\right) \, dx.
\end{eqnarray}
%By $\mathcal{N}(0)$ we have denoted the family of all neighborhoods of $0$ in the strong $L^2$ topology.
%The following claims and assumptions are the same as in \cite{Velcic-13}.
%\begin{remark}\label{minat1}
%Since the above expressions are monotonly decreasing in  $\mathcal{N}(0)$ it is enough to take the supremum on the monotone sequence of neighborhoods that shrinks to $\{0\}$ e.g. the sequence of (open or closed) balls of radius $r$, when $r \to 0$.
%\end{remark}
\begin{remark}\label{minat2}
By using standard diagonalization argument  it can be shown that for any $(h_n)_{n\in \N}$ monotonly decreasing to $0$ the infimum in expressions \eqref{gli} and \eqref{gls} are attained.
\end{remark}

\begin{lemma} \label{lem:ocjena}
There exists a constant $C>0$ dependent only on $\eta_1,\eta_2$ such that for every sequence $(h_n)_{n \in \N}$ monotonly decreasing to $0$ and $A \subset [0,L]$ open set the following inequality is valid
\begin{eqnarray}\label{ocjena1111}
\left| K^{-}_{(h_n)_{n \in \N}}(m_1,A)-K^{-}_{(h_n)_{n \in \N}}(m_2,A)\right|&\leq& C \|m_1-m_2\|_{L^2}\left(\|m_1\|_{L^2}+\|m_2\|_{L^2}\right),\\\nonumber & &  \ \forall m_1,m_2 \in L^2(\Omega,\R^3),
\end{eqnarray}
The analogous claim holds for $K^{+}_{(h_n)_{n \in \N}}$.
\end{lemma}
\begin{proof}
The proof goes in an analogous way as the proof of \cite[Lemma 3.5]{Velcic-13}.
\end{proof}

If $A$ and $B$ are subsets of $[0,L]$, we denote by $A\ll B$ if $\bar{A}$ is compact and contained in $B$. The following definitions are standard for $\Gamma$-convergence techniques (see \cite{DM93}).

\begin{definition}
We say that a family of sets $\mathcal{D}$ of $\mathcal{A}$ is dense in the family $\mathcal{A}$ if for every $A,B  \in \mathcal{A}$, with $A \ll B$, there exists $D \in \mathcal{D}$, such that $A\ll D \ll B$.
\end{definition}

Let $\mathcal{D}$ denote the countable family of open subsets of $[0,L]$ which is dense in the class $\mathcal{A}$ of all open subsets of $[0,L]$ and such that every $D \in \mathcal{D}$ is a finite union of open intervals which are subsets of $[0,L]$.

By using previous lemma and diagonal procedure we can also easily argument the following claim.
\begin{lemma} \label{podniz}
For every sequence $(h_n)_{n \in \N}$ monotonly decreasing to zero there exists a subsequence, still denoted by $(h_n)_{n \in \N}$, such that
$$ K^+_{(h_n)_{n \in \N}}(m, D)=K^{-}_{(h_n)_{n \in \N}} (m,D),\quad \forall m \in L^2(\Omega,\R^{3}), \ \forall D \in \mathcal{D}. $$
\end{lemma}
We will now make an assumption on the sequence $(h_n)_{n \in \N}$ monotonly decreasing to zero and family $(Q^{h_n})_{n \in \N}$.
\begin{assumption}
  \label{ass:main}
  For given $(h_n)_{n \in \N}$ monotonly decreasing to zero we suppose that  we have
  $$K^{+}_{(h_n)_{n \in \N}}(m,D)=K^{-}_{(h_n)_{n \in \N}}(m,D)=:K(m,D),\quad \forall m \in L^2(\Omega,\R^{3}), \ \forall D \in \mathcal{D}. $$
\end{assumption}
Although the numbers $K(m,D)$ also depend on the sequence, we will not write it, since it will be clear from the context on which sequence we are referring to.
\begin{remark}
As in \cite[Lemma 3.8]{Velcic-13} we can see that if a sequence $(h_n)_{n \in \N}$ satisfies the Assumption \ref{ass:main} than we have that
$$K^{+}_{(h_n)_{n \in \N}}(m,O)=K^{-}_{(h_n)_{n \in \N}}(m,O)=:K(m,O),\quad \forall m \in L^2(\Omega,\R^{3}), \ \forall O \subset [0,L] \textrm{ open}. $$
\end{remark}

%We introduce the space of matrix fields which appear as limit strains in bending model
%$$\mathcal{S}_{B}([0,L])=\{\sym \iota(A(0,x_2,x_3)^t) : A \in L^2([0,L];\mathbb{M}^3_{\skw}) \}. $$

%\begin{remark} \label{razjasnjenje}
%Starting from Lemma \ref{podniz} and Assumption \ref{ass:main} we could state the results, only restricting %ourselves on the space $\mathcal{S}_{B}([0,L])$ instead of $L^2(\Omega,\R^3)$. We refrained ourselves from doing %so for the sake of generality, when it was meaningfull.
%\end{remark}

The following lemma is analogous with \cite[Lemma 3.10]{Velcic-13}. We shall not prove it here.
\begin{lemma}\label{lem:improveglavnazamjena}
Let  $(h_n)_{n \in \N}$ be a sequence monotonly decreasing to $0$ which satisfies Assumption \ref{ass:main}.
 Take $m \in  L^2(\Omega,\R^3)$ and $O \subset \omega$ open.
Then there exists a subsequence $(h_{n(k)})_{k\in \N}$ and $(\vartheta_k)_{k \in \N} \subset W^{1,2}(O \times \omega,\R^3)$ such that
\begin{enumerate}[(a)]
\item
$(\vartheta_{k,1},h_{n(k)}  \vartheta_{k,2},h_{n(k)} \vartheta_{k,3}) \to 0$, $\int_{\omega} x_3 \vartheta_{k,2} \to 0$ strongly in $L^2$,
\item
 $(|\sym \nabla_{h_{n(k)}} \vartheta_k|^2 )_{k \in \N}$ is equi-integrable,
 $$ \sym \nabla_{h_{n(k)}} \vartheta_k=\sym \iota((A_k)' d_{\omega})+\sym \nabla_{h_{n(k)}} v_{k}. $$
 Here  $(A_k)_{k \in \N} \subset  W^{1,2}([0,L];\mathbb{M}^3_{\skw})$, $A_k \to 0$ strongly in $L^2$, $(v_k)_{k \in \N} \subset W^{1,2}(\Omega;\R^3)$, $v_k \to 0$ strongly in $L^{2}$. Moreover we have
 where $\left(|(A_k)'|^2\right)_{k\in\N}$ and $\left(| \nabla_{h_{n(k)}} v_{k}|^2\right)_{k \in \N}$ are equi-integrable.
 Also the following is valid
 $$\limsup_{k \to \infty} \left( \|A_k\|_{W^{1,2}(O)}+\|\nabla_{h_{n(k)}}v_k \|_{L^2(O \times \omega)}\right)\leq C \left( \eta_2 \|m\|^2_{L^2}+1 \right),$$
 where $C$ is independent of the domain $O$ and for each $k \in \N$ we have  $A_k=0$ in a neighborhood of $\partial O$ and $v_k=0$ in a neighborhood of $\partial O \times \omega$.
\item $$K(m,O)=\lim_{k\to \infty} \int_{O \times \omega} Q^{h_{n(k)}}(x,\iota(m)+\nabla_{h_{n(k)}} \vartheta_k) \ud x.$$

\end{enumerate}
\end{lemma}

Let $\mathcal{L}^h:\Omega \times \mathbb{M}^3 \to \mathbb{M}^3_{\sym}$ be a
measurable mapping such that for every $x \in \Omega$, $\mathcal{L}(x, \cdot)$
is a unique positive semidefinite linear operator such that
$Q^h (x,M)=\mathcal{L}^h(x, M) \cdot M$, for all $M \in \mathbb{M}^3$.

Notice that
\begin{equation}\label{svojstvoL}
\mathcal{L}^h (x, M)= \mathcal{L}^h (x, \sym M), \quad \| \mathcal{L}^h\|_{L^\infty} \leq \eta_2.
\end{equation}
\begin{corollary}\label{pommmm}
Take $m \in L^2(\Omega;\R^{3})$ and a sequence $(h_n)_{n \in \N}$ monotonly decreasing to $0$ that satisfy Assumption \ref{ass:main}
and for which there exists $(\vartheta_n)_{n \in \N} \subset W^{1,2}(\Omega,\R^3)$ such that
\begin{enumerate}[(a)]
\item
$(\vartheta_{n,1},h_{n}  \vartheta_{n,2},h_{n} \vartheta_{n,3}) \to 0$, $\int_{\omega} x_3 \vartheta_{n,2} \to 0$ strongly in $L^2$,
%\item $(|\sym \nabla_{h_{n}} \vartheta_n|^2 )_{n \in \N}$ is equi-integrable,
\item $\displaystyle K(m,[0,L])=\lim_{n\to \infty} \int_{\Omega} Q^{h_{n}}(x,\iota(m)+\nabla_{h_{n}} \vartheta_n) \ud x.$
\end{enumerate}
Then we have that:
\begin{enumerate}[(I)]
\item $(|\sym \nabla_{h_n} \vartheta_n|^2)_{n\in \N}$ is equi-integrable;
\item for every $O$  open subset of $[0,L]$ we have that
\begin{equation} \label{jednakost1}
K(m,O)=\lim_{n\to \infty} \int_{O \times \omega} Q^{h_{n}}(x,\iota(m)+\nabla_{h_{n}} \vartheta_n) \ud x;
\end{equation}
\item If $(\psi_n)_{n\in \N} \subset W^{1,2}(\Omega;\R^3)$ is any other sequence that satisfies (a) and (b) then
$$\|\sym \nabla_{h_n} \psi_n - \sym \nabla_{h_n} \vartheta_n\|_{L^2} \to 0. $$
  and $(|\sym \nabla_{h_n} \psi_n|^2)_{n\in \N}$ is equi-integrable.
\end{enumerate}
\end{corollary}
\begin{proof}
From (Q1) and by taking the zero subsequence we obtain the bound
\begin{equation} \label{bound}
\limsup_{n \to \infty} \|\sym \nabla_{h_{n}}\vartheta_n \|_{L^2(\Omega)}\leq C \left( \eta_2 \|m\|^2_{L^2}+1 \right).
\end{equation}
From Corollary \ref{cormaroje} there are  sequences $(A_n)_{n \in \N} \subset  W^{1,2}((0,L);\mathbb{M}^3_{\skw})$ and $(v_n)_{n \in \N} \subset W^{1,2}(\Omega;\R^3)$ such that $A_n \to 0$ and  $v_n \to 0$ strongly in $L^{2}$ and
$$\left\|\sym \nabla_{h_{n}} \vartheta_n-\sym \iota((A_n)' d_{\omega})-\sym \nabla_{h_{n}} v_{n}\right\|_{L^2} \to 0.$$
From (\ref{ocjenaahvh}) we  obtain that
 $$\limsup_{k \to \infty} \left( \|A_n\|_{W^{1,2}(\Omega)}+\|\nabla_{h_{n}}v_n \|_{L^2(\Omega)}\right)\leq C \left( \eta_2 \|m\|^2_{L^2}+1 \right).$$

To prove that  $(|\sym \nabla_{h_{n}} \vartheta_n|^2 )_{n \in \N}$ is equi-integrable, let us assume the opposite, i.e., that there is $\eps>0$ such that for every $k>0$ there is a measurable set $S_k$ such that $\vert S_k\vert < \frac{1}{k} $ and there is a $n(k)>n(k-1)$ such that
\begin{equation*}
\int_{S_k} | \sym \nabla_{h_{n(k)}} \vartheta_{n(k)}|^2 \ud x \geq \eps.
\end{equation*}

On the other hand, by Lemma \ref{prvavazna} and \ref{drugavazna}  there is a subsequence, still denoted by $n(k)$ and sequences $(\tilde{A}_{k})_{k \in \N} \subset W^{1,2}((0,L);\mathbb{M}^3_{\skw})$ and $(\tilde{v}_{k})_{k \in \N} \subset W^{1,2}(\Omega;\R^3)$ such that:
\begin{enumerate}
\item[(i)] $\displaystyle \lim_{k \to \infty } \left\vert \Omega \cap \left\{ \tilde{A}_{k} \neq A_{n(k)} \mbox{ or }    \tilde{A}_{k}' \neq A_{n(k)}' \right\}\right\vert =0,$
\item[(ii)] $ \lim_{k \to \infty } \left\vert \Omega \cap \left\{ \tilde{v}_{k} \neq v_{n(k)} \mbox{ or }    \nabla \tilde{v}_{k} \neq \nabla v_{n(k)}  \right\}\right\vert =0,$
\item[(iii)] $\tilde{A}_{k}'$ and $\nabla_{h_{n(k)}}\tilde{v}_{k}$ are equi-integrabile.
\end{enumerate}

Now since
\begin{align*}
K(m,[0,L])&=\liminf_{k\to\infty} \int_{\Omega} Q^{h_{n(k)}}\left(x,\iota(m) + \nabla_{h_{n(k)}}\vartheta_{n(k)}\right) \ud x \\
&> \liminf_{k\to\infty} \int_{\Omega} \chi_{\Omega\setminus S_k} Q^{h_{n(k)}}\left(x,\iota(m) + \nabla_{h_{n(k)}}\vartheta_{n_k}\right) \ud x   \\
&=  \liminf_{k\to\infty} \int_{\Omega}  \chi_{\Omega\setminus S_{k}} Q^{h_{n(k)}}\left(x,\iota(m) + \sym\iota((\tilde{A}_k)'d_{\omega}) + \sym{\nabla_{h_{n(k)}} \tilde{v}_k}\right) \ud x  \\
&= \liminf_{k\to\infty} \int_{\Omega}  Q^{h_{n(k)}} \left(x,\iota(m) + \sym \iota((\tilde{A}_k)'d_{\omega}) + \sym{\nabla_{h_{n(k)}} \tilde{v}_k}\right) \ud x \\
&=K(m,[0,L]),
\end{align*}
which gives a contradiction. Therefore, $(|\sym \nabla_{h_{n}} \vartheta_n|^2 )_{n \in \N}$ is equi-integrabile.

 We will show that $\vartheta_n$ is optimal on any open set $O \in \mathcal{D}$ which is a finite union of disjoint open intervals.
 If that was wrong then there would exist a subsequence, still denoted by $(h_n)_{n \in \N}$
  such that  there is a sequence  $(\psi_n^1)_{n \in \N} \subset W^{1,2}(O\times \omega,\R^3)$  satisfying the conditions of Lemma \ref{lem:improveglavnazamjena}  and
\begin{equation*}
K(m,O)=\lim_{n \to \infty} \int\limits_{O \times \omega} Q^{h_{n}}\left(x,\iota(m)+\nabla_{h_{n}}\psi_n^1\right) \ud x < \lim_{n \to \infty} \int\limits_{O\times \omega} Q^{h_{n}}\left(x,\iota(m)+\nabla_{h_{n}}\vartheta_n \right) \ud x.
\end{equation*}
On the other hand,
on the further subsequence, still denoted by $(h_n)_{n \in \N}$
we take the sequence  $\psi_n^2 \subset W^{1,2}([0,L]\setminus\bar{O},\R^3)$
satisfying the conditions of Lemma \ref{lem:improveglavnazamjena}  and
\begin{eqnarray*}
K(m,(0,L) \backslash \bar{O} )&=&\lim_{n \to \infty} \int\limits_{((0,L)\setminus \bar{O}) \times \omega} Q^{h_n}\left(x,\iota(m)+\nabla_{h_n}\psi_n^2\right) \ud x\\ &\leq& \lim_{n \to \infty} \int\limits_{([0,L]\setminus \bar{O})\times \omega} Q^{h_n}\left(x,\iota(m)+\nabla_{h_{n}}\vartheta_n^2\right) \ud x.
\end{eqnarray*}
By using Lemma \ref{peter1} we define $(\psi_n)_{n \in \N} \subset W^{1,2}(\Omega;\R^3)$
such that
$$ \sym \nabla_{h_n} \psi_n=\chi_{O} \sym \nabla_{h_n} \psi_n^1+\chi_{[0,L] \backslash \bar{O}} \sym \nabla_{h_n} \psi_n^2. $$
We conclude that
\begin{align*} \lim_{n \to \infty}\int\limits_{[0,L]\times \omega} Q^{h_n}\left(x,\iota(m)+\nabla_{h_n}\psi_n\right) \ud x &< \lim_{n \to \infty}\int\limits_{[0,L]\times \omega} Q^{h_n}\left(x,\iota(m)+\nabla_{h_n}\vartheta_n\right) \ud x\\ &=K(m,[0,L]),
\end{align*}
which yields a contradiction with the optimality of the sequence $(\vartheta_n)_{n \in \N}$.

Now for any open $O\subset [0,L]$, by density, there is an increasing family of sets $(D_k)_{k \in \N} \subset \mathcal{D}$ which exhausts $O$. Since $(\vartheta_n)_{n \in \N}$ is optimal on each $D_k$ and since $K(m,O) \geq K(m,D_k)$ (this can be easily seen from Lemma \ref{lem:improveglavnazamjena}) we deduce from equi-integrability of $\left(|\sym{\nabla_{h_n}\vartheta_{n}}|^2\right)_{n \in \N}$ that
\begin{align*}
K(m,O)\geq \lim_{k \to \infty} K(m,D_k)= \lim_{n \to \infty} \int_{O \times \omega} Q^{h_n}\left(x,\iota(m)+\nabla_{h_n}\vartheta_n\right)  \ud x.
\end{align*}
that $\vartheta_k$ is also optimal for $K(m,O)$ and (II) is proved.

To prove (III)  we first note that
 \begin{equation} \label{operationalform}
 \lim_{n\to \infty}\int_{\Omega} \mathcal{L}^h (x, \iota(m)+\nabla_{h_n} \vartheta_n) \cdot \nabla_{h_n} \tilde{\psi}_n =0,
 \end{equation}
for every $(\tilde{\psi_n})_{n \in \N} \subset W^{1,2}(\Omega;\R^3)$ that satisfies (a) and such that $\vert \sym{\nabla_{h_n} \tilde{\psi}_n}\vert$ is bounded in $L^2$.

To prove this we take $\eps > 0$  and for $k$ large enough we derive:
\begin{align*}
 0 &\leq  \int_{\Omega} Q^{h_n} (x,\iota(m)+\nabla_{h_n} \vartheta_n + \varepsilon \nabla_{h_n} \tilde{\psi}_n ) \ud x - \int_{\Omega} Q^{h_n} (x,\iota(m)+\nabla_{h_n} \vartheta_n) \ud x\\
 &= \int_{\Omega} \mathcal{L}^h (x, \iota(m)+\nabla_{h_n} \vartheta_n+ \varepsilon \nabla_{h_n} \tilde{\psi}_n) \cdot (\iota(m)+\nabla_{h_n} \vartheta_n+ \varepsilon \nabla_{h_n} \tilde{\psi}_n) \ud x \\
 &- \int_{\Omega} \mathcal{L}^h (x, \iota(m)+\nabla_{h_n} \vartheta_n) \cdot (\iota(m)+\nabla_{h_n} \vartheta_n ) \ud x \\
 &= 2\varepsilon \int_{\Omega} \mathcal{L}^h (x, \iota(m)+\nabla_{h_n} \vartheta_n) \cdot ( \nabla_{h_n} \tilde{\psi}_n) \ud x  +\eps^2 \int_{\Omega} \mathcal{L}^h (x,  \nabla_{h_n} \tilde{\psi}_n) \cdot ( \nabla_{h_n} \tilde{\psi}_n) \ud x \\
  &\leq 2\varepsilon \int_{\Omega} \mathcal{L}^h (x, \iota(m)+\nabla_{h_n} \vartheta_n) \cdot ( \nabla_{h_n} \tilde{\psi}_n) \ud x  +\eps^2 \eta_2 \vert \sym{\nabla_h \tilde{\psi}_n}\vert^2\\
  &= 2\varepsilon \int_{\Omega} \mathcal{L}^h (x, \iota(m)+\sym\nabla_{h_n} \vartheta_n) \cdot ( \sym\nabla_{h_n} \tilde{\psi}_n) \ud x  +\eps^2 \eta_2 \vert \sym{\nabla_h \tilde{\psi}_n}\vert^2.
\end{align*}
If  (\refeq{operationalform}) didn't hold we would choose  $\eps$  (by taking the appropriate sign) such that the linear term dominates and the inequality is violated. Thus, we deduce
(\refeq{operationalform}), by the contradiction.
 To prove the last claim we take two sequences  $(\vartheta_n)_{n\in \N} \subset W^{1,2}(\Omega;\R^3)$, $(\psi_n)_{n\in \N} \subset W^{1,2}(\Omega;\R^3)$ that satisfy (a) and (b). Now we have, using \eqref{operationalform}
\begin{eqnarray*}
\eta_1\|\sym \nabla_{h_n} (\psi_n- \vartheta_n)\|^2_{L^2} &\leq& \int_{\Omega} \mathcal{L}^{h_n}(x, \nabla_{h_n} (\psi_n- \vartheta_n)) \cdot \nabla_{h_n} (\psi_n- \vartheta_n) \ud x \\
&=& \int_{\Omega} \mathcal{L}^{h_n}(x,\iota(m)+ \nabla_{h_n} \psi_n) \cdot \nabla_{h_n} (\psi_n-\vartheta_n) \ud x\\ & &-\int_{\Omega} \mathcal{L}^{h_n}(x,\iota(m)+ \nabla_{h_n} \vartheta_n) \cdot \nabla_{h_n} (\psi_n-\vartheta_n) \ud x\to 0.
\end{eqnarray*}

\end{proof}
\begin{comment}
\begin{remark} \label{iskaz1}
From the proof of the previous corrolary we can observe that if  $(\vartheta_n)_{n \in \N} \subset W^{1,2}(\Omega,\R^3)$ is such that it satisfies (a) and (b) of the previous corollary and
$ (\psi_n)_{n \in \N} \subset W^{1,2}(\Omega,\R^3)$ is such that it satisfies (a) of the corollary and
\begin{equation} \label{operationalform2}
 \lim_{n \to \infty} \int_{\Omega} \mathcal{L}^{h_n} (x, \iota(m)+\nabla_{h_n} \psi_n) \cdot \nabla_{h_n} \tilde{\psi}_n =0,
 \end{equation}
for every $(\tilde{\psi}_n)_{n \in \N} \subset W^{1,2}(\Omega,\R^3)$ that also satisfies (a) of the corollary then necessarily
$$\|\sym \nabla_{h_n} \psi_n - \sym \nabla_{h_n} \vartheta_n\|_{L^2} \to 0. $$
Notice also that \eqref{operationalform2} implies that $\limsup_{n \to \infty} \|\sym \nabla_{h_n} \psi_n\|_{L^2} < \infty$, which can be seen by putting $\tilde{\psi}_n=\psi_n$ in \eqref{operationalform2}.
\end{remark}
\end{comment}
The following lemma proves the compactness result we need.
\begin{lemma}
For every sequence $(h_n)_{n \in \N}$ that satisfy the Assumption \ref{ass:main} there exists a subsequence, still denoted by $(h_n)_{n \in \N}$ such that
for each $m \in L^2(\Omega;\R^3)$ there exists $\left(\vartheta_n(m)\right)_{n \in \N} \subset W^{1,2}(\Omega;\R^3)$ which satisfies
\begin{enumerate}[(a)]
\item
$\left(\vartheta_{n,1}(m),h_{n}  \vartheta_{n,2}(m),h_{n} \vartheta_{n,3}(m)\right) \to 0$, $\int_{\omega}x_3\vartheta_{n,2}(m) \to 0$ strongly in $L^2$,
%\item $(|\sym \nabla_{h_{n}} \vartheta_n|^2 )_{n \in \N}$ is equi-integrable,
\item $$K(m,[0,L])=\lim_{n\to \infty} \int_{\Omega} Q^{h_{n}}(x,\iota(m)+\nabla_{h_{n}} \vartheta_n(m))\, \ud x.$$
\end{enumerate}
\end{lemma}
\begin{proof}
Let $\mathcal{M} \subset L^2(\Omega;\R^3)$ be a countable dense family. By diagonalization procedure it is possible to construct the subsequence, still denoted by $(h_n)_{n \in \N}$ such that for each $m \in  \mathcal{M}$ there is a sequence $\vartheta(m)_n$ for which (a) and (b) holds. Now we take the sequence
$(m_{n})_{n \in \N} \subset \mathcal{M}$ such that $m_n  \to m$ in $L^2$ as $n \to \infty$ and define the  strictly increasing function $k:\N \to \N$ which satisfies for every $n_0 \in \N$
\begin{eqnarray*}
\left\vert K(m_{n_0},[0,L]) - \int_{\Omega} Q^{h_n}\left(x,\iota(m_{n_0}) + \nabla_{h_n} \vartheta_{n}(m_{n_0}) \right) \ud x \right\vert &<& \frac{1}{n_0},\\
\left\| \left(\vartheta_{n,1}(m_{n_0}),h_{n}  \vartheta_{n,2}(m_{n_0}),h_{n} \vartheta_{n,3}(m_{n_0})\right) \right\|_{L^2} &<& \frac{1}{n_0}, \\
\left\|\int_{\omega}x_3\vartheta_{n,2}(m_{n_0})\right\|_{L^2} &<& \frac{1}{n_0}, \quad
 \mbox{ for every }  n \geq k(n_0).
\end{eqnarray*}
 For every $i \in \N $ and $j \in [k(i),k(i+1))$ take  $\vartheta_j(m):= \vartheta_j(m_{k(i)})$ and  use Lemma \ref{lem:ocjena} to show (b).
\end{proof}
We are now in position to make the assumption on the family $(Q^h)_{h>0}$.
\begin{assumption}\label{konacna}
We assume that for every $m \in L^2(\Omega;\R^3)$ and every $O \subset [0,L]$ open there exists number $K(m,O)$ such that for every $(h_n)_{n \in \N}$ monotonly decreasing to zero there exists $(\vartheta_n(m))_{n \in \N} \subset W^{1,2}(\Omega;\R^3)$ such that
\begin{enumerate}[(a)]
\item
$\left(\vartheta_{n,1}(m),h_n  \vartheta_{n,2}(m),h_n \vartheta_{n,3}(m)\right) \to 0$, $\int_{\omega} x_3 \vartheta_{n,2}(m) \to 0$ strongly in $L^2$,
%\item $(|\sym \nabla_{h_{n}} \vartheta_n|^2 )_{n \in \N}$ is equi-integrable,
\item $$K(m,O)=\lim_{n \to \infty} \int_{O \times \omega} Q^{h_n}(x,\iota(m)+\nabla_{h_n} \vartheta_n(m))\ud x.$$
\end{enumerate}
Moreover, we have
\begin{eqnarray*}
& &\\ \nonumber K(m,O)&=&\min \Big\{\liminf_{n \to \infty}  \int_{O \times \omega} Q^{h_n}\left(x,\iota(m)+\nabla_{h_n} \psi^{h_n}\right) \ud x: \\ \nonumber & & \hspace{-12ex} (\psi_1^{h_n},h_n\psi_2^{h_n},h_n\psi_3^{h_n}) \to 0 \textrm{ strongly in } L^2(O \times \omega;\R^3)  \int_{\omega} x_3 \psi^{h_n}_2 \to 0 \textrm{ strongly in } L^2(O) \Big\}
%\\ %\nonumber
%&=& \sup_{\mathcal{U} \subset \mathcal{N}(0)} \liminf_{n \to \infty}\inf_{\psi \in H^1(O \times I,\R^3) \atop (\psi_1,h_n\psi_2,h_n\psi_3) \in \mathcal{U}}
 %\int_{O \times \omega} Q^{h_n}\left(x,\iota(m)+\nabla_{h_n} \psi\right) \, dx,
\\ &=&\min \Big\{\limsup_{n \to \infty} \int_{A \times I} Q^{h_n}\left(x,\iota(m)+\nabla_{h_n} \psi^{h_n}\right) \ud x:\\ \nonumber & & \hspace{-12ex}  (\psi_1^{h_n},h_n\psi_2^{h_n},h_n \psi_3^{h_n}) \to 0 \textrm{ strongly in } L^2(O \times \omega;\R^3), \ \int_{\omega} x_3 \psi^{h_n}_2 \to 0 \textrm{ strongly in } L^2(O)  \Big\}.
%\\ \nonumber
%&=& \sup_{\mathcal{U} \subset \mathcal{N}(0)} \limsup_{n \to \infty} \inf_{\psi \in H^1(O \times I,\R^3) \atop %(\psi_1,h_n\psi_2,h_n \psi_3) \in  \mathcal{U}} \int_{O \times I} Q^{h_n}\left(x,\iota(m)+\nabla_{h_n} %\psi\right) \, dx.
\end{eqnarray*}
\end{assumption}
We define the mapping $m: L^2([0,L];\mathbb{M}^3_{\skw}) \times L^2([0,L]) \to L^2([0,L];\R^3)$ by $m(A,a)=A(0,x_2,x_3)^t+a$.
The following proposition is the analogous to \cite[Proposition 2.9]{Velcic-13}.
\begin{proposition}\label{identi}
Let Assumption \ref{konacna} be valid.
There exists a measurable function $Q:[0,L] \times \mathbb{M}^3_{\skw}\times \R \to \R$ such that for every
$O \subset [0,L]$ open and every $A \in L^2([0,L];\mathbb{M}^3_{\skw})$
 we have
\begin{equation} \label{integralineq}
 K(m(A,a),O)=\int_{O} Q(x_1,A(x_1),a(x_1)) \ud x_1.
\end{equation}
Moreover, $Q$ satisfies the following property
  \item[(Q'1)] for  almost all $x_1\in [0,L]$ the map $Q(x_1,\cdot,\cdot)$ is a quadratic form and there is a positive constant $C_{\omega}$ such that
\begin{equation}\label{boundQ}
      C_{\omega}(|A|^2+|a|^2) \leq Q (x_1,A,a)\leq \eta_2\left(\max\{\mu_2,\mu_3\}|A|^2+|a|^2\right) \quad \text{for all $(A,a) \in \mathbb{M}^3_{\skw} \times \R$.}
    \end{equation}
\end{proposition}
\begin{proof}
The existence of $Q$  and  the proof of \ref{integralineq} is identitical as in \cite{Velcic-13}. Therefore, we will only prove the boundedness and coercivity property.
The function $Q$ is defined via (see \cite{Velcic-13})
\begin{equation}\label{defQ1}
Q(\bar{x}_1,A,a) =\lim_{r \to 0}\frac{1}{2r}K\left(m(A,a),B(\bar{x}_1,r)\right), \textrm{for a.e. } \bar{x}_1 \in [0,L].
\end{equation}
The upper bound in (\ref{boundQ}) is easily obtained by taking the zero subsequence $\vartheta_n=0$ and by using (Q1) and (\ref{normalizacija}) to deduce
$$\vert Q(\bar{x}_1,A,a) \vert \leq \eta_2 \left( \vert \sym \iota(a + Ad_{\omega}) \vert^2 \right) \leq  \eta_2 \left(  \max{ \left\{\mu_2,\mu_3\right\}} \vert A   \vert^2 + \vert a \vert^2 \right),$$
for a.e. $\bar{x}_1 \in (0,L)$.

From the Assumption \ref{konacna} and Corollary \ref{cormaroje} we deduce that there are bounded sequences  $(A^h)_{h>0} \subset  W^{1,2}([0,L];\mathbb{M}^3_{\skw})$ and $(z^h)_{h>0} \subset W^{1,2}(\Omega;\R^3)$ such that $A^h \to 0$ and $v^h \to 0$ strongly in $L^{2}$ and
$$K(m(A,a),B(\bar{x}_1,r))=\lim_{h \to 0}\int_{B(\bar{x}_1,r) \times \omega}  Q^{h} \left(x,\iota(m) + \sym \iota((A^h)'d_{\omega}) + \sym{\nabla_{h} v^h}\right) \ud x. $$
%Moreover we have
%\begin{equation}
%\|A^h\|_{W^{1,q}}+\|v^h \|_{L^{q}}+\|\nabla_h v^h \|_{L^{q}}\leq %C(\omega) \|\sym \nabla_h u^h \|_{L^q}.
%\end{equation}
for some $C>0$. We can assume, by the density argument, that $v^h$ and $A^h$ are smooth functions.
Using the property (Q1) we have
\begin{align*}
K(m(A,a),B(\bar{x}_1,r)) \geq  \eta_1 (I_1 + I_2),
\end{align*}
where $I_1$ and $I_2$ are defined by:
\begin{align*}
I_1 &= \lim_{h\to 0} \int\limits_{B(\bar{x}_1,r)\times \omega} \left( a + A_{12}x_2 + A_{13}x_3 + (A^h_{12})'x_2 + (A^h_{13})'x_3 +\partial_1{v_1^h} \right)^2 dx \\
I_2 &= \frac{1}{2}\lim_{h\to 0} \left\{ \int\limits_{B(\bar{x}_1,r)\times \omega} \left( A_{23}x_3 + (A_{23}^h)' x_3 + \partial_1{v_2^h}+ \frac{\partial_{2}{v_1^h}}{h} \right)^2 dx  \right.  \\
 &  \left. + \int\limits_{B(\bar{x}_1,r)\times \omega} \left( - A_{23}x_2 -(A_{23}^h)' x_2 +\partial_1{v_3^h} + \frac{\partial_{3}{v_1^h}}{h} \right)^2 dx  \right\}
\end{align*}
From the choice of the coordinate axis, see (\ref{normalizacija}), we have that for every $x_1 \in B(\bar{x}_1,r)$
$$\int\limits_{\{x_1\} \times\omega}  a A_{12}x_2  \ud x_2 \ud x_3 =  \int\limits_{\{x_1\} \times\omega}  a A_{13}x_3 \ud x_2 \ud x_3  =
\int\limits_{\{x_1\} \times \omega}  A_{13}A_{12}x_2 x_3 \ud x_2 \ud x_3  =0.
$$
Thus, we derive that
\begin{align*}
I_1 & \geq \int_{B(\bar{x}_1,r)\times \omega}\left( \vert a \vert^2 + x_2^2 A_{12}^2 + A_{13}x_3^2 \right) \ud x \\ + & 2\lim_{h \to 0} \int_{B(\bar{x}_1,r)\times \omega} \left( a + A_{12}x_2 + A_{13}x_3\right)\left( (A^h_{12})'x_2 + (A^h_{13})'x_3 +\partial_1{v_1^h} \right) \ud x .
\end{align*}
Since
$(A_h)' \rightharpoonup 0$ and $\partial_1{v_1}^h \rightharpoonup 0$ weakly in $L^2$  the mixed term vanishes as $h \to 0$.
Hence, we obtain that
\begin{equation}\label{ocjenaI1}
I_1 \geq 2r \left( \vert a \vert^2 + \mu_2 \vert A_{12}\vert^2 +\mu_3 \vert A_{13}\vert^2 \right).
\end{equation}
To obtain the lower bound for $I_2$ we
we look for a solution of the minimum problem
$$\min_{\psi \in H^1(\omega)} \int_{\omega} \vert u - \nabla \psi \vert^2 \ud x.$$
The solution of the problem is unique up to constant and  satisfies the variational equation
\begin{equation}\label{project}
\int_{\omega} \left(\nabla \varphi_u  - u\right)\cdot \nabla \psi \ \ud x = 0,
\end{equation}
for every $\psi \in H^1(\omega) $.
The solution corresponds to $L^2$ projection on the space
\begin{equation*}\label{projequation}
G(\omega)= \left\{w \in L^2(\omega;\R^2): w=\nabla p, \mbox{ for some } p \in H^1(\omega)  \right\},
\end{equation*}
which is a closed subspace in $L^2(\omega;\R^2)$.
%By the projection theorem, for each $u \in L^2(\omega;\mathbb{R}^2)$ there is a unique function $\varphi_u  \in H^1(\omega)$ such that $\varphi_u$  a solution of the minimum problem:
 We denote with $Pu = u - \nabla \varphi_u $.
%Note that by taking $\psi = \varphi_u$ we deduce from (\ref{projequation}) that:
 %$$\Vert u \Vert_{L^2(\omega)} \geq \Vert \nabla \varphi \Vert_{L^2(\omega)}$$
Denote also with
$$\Psi^h(x)= \left( A _{23}+A_{23}^h \right)
\left(
\begin{array}{c}
x_3 \\
-x_2
\end{array}
\right) +
\left(
\begin{array}{c}
\partial_1{v_2^h}\\
\partial_1{v_3^h}
\end{array}
\right)
 +
 \frac{1}{h}
\left(
\begin{array}{c}
\partial_2{v_1^h}\\
\partial_3{v_1^h}
\end{array}
\right)
$$
we have that
$$I_2 =\int\limits_{B(\bar{x}_1,r)\times \omega} \vert \Psi^h \vert^2  dx \geq \int\limits_{B(\bar{x}_1,r)\times \omega} \vert P\Psi^h \vert^2  dx,$$
where $P \Psi^h$ equals
$$P(\Psi^h(x))= \left( A _{23}+A_{23}^h \right)
P\left(
\begin{array}{c}
x_3 \\
-x_2
\end{array}
\right) +
P
\left(
\begin{array}{c}
\partial_1{v_2^h}\\
\partial_1{v_3^h}
\end{array}
\right).
$$
Notice that the projection is done for fixed $x_1 \in [0,L]$.
This yields that:
\begin{align*}
I_2 & \geq \bar{C}_{\omega} \lim_{h \to 0} \left( \int_{B(\bar{x}_1,r) } \vert A_{23} \vert^2 +2
\int_{B(x_1,r)} A_{23} (A_{23}^h)'  dx_1\right) \\
&+2 \lim_{h \to 0} \int_{B(\bar{x}_1,r) \times \omega } A_{23} P \left(
\begin{array}{c}
x_3 \\
-x_2
\end{array}
\right)
\cdot
P\left(
\begin{array}{c}
\partial_1{v_2^h}\\
\partial_1{v_3^h}
\end{array}
\right)   dx
\end{align*}
where the constant $\bar{C}_{\omega}$ equals
\begin{equation}\label{defC}
\bar{C}_{\omega}=\int_{\omega} \left\vert P\left(\begin{array}{c}
x_3 \\
-x_2
\end{array}
\right)\right\vert^2 \ud x_2 \ud x_3.
\end{equation}
Since $A_h' \rightharpoonup  0$ in $L^2$ the second term converges to zero.
Since $P$ is the projection we have that
\begin{align*}
\int \limits_{B(\bar{x}_1,r) \times \omega} A_{23}P\left(
\begin{array}{c}
x_3 \\
-x_2
\end{array}
\right)
\cdot
P\left(
\begin{array}{c}
\partial_1{v_2^h}\\
\partial_1{v_3^h}
\end{array}
\right)   \ud x = \int \limits_{B(\bar{x}_1,r) \times \omega} A_{23}P\left(
\begin{array}{c}
x_3 \\
-x_2
\end{array}
\right)
\cdot
\left(
\begin{array}{c}
\partial_1{z_2^h}\\
\partial_1{z_3^h}
\end{array}
\right)   \ud x \to 0,
\end{align*}

since $\partial_1{z}^h \rightharpoonup 0$ weakly in $L^2$.  We obtain that
$$I_2 \geq 2r \bar{C}_{\omega}  A_{23}^2.$$
Combing this with \eqref{defQ1} and (\ref{ocjenaI1}) and taking the limit as $r \to 0$ yields the coercivity of $Q$.
\end{proof}

We define the function  $Q_0: [0,L] \times \mathbb{M}^3_{\skw} \to \R$ such that
\begin{equation}
Q_0(x_1,A) =\min_{a \in \R} Q(x_1,A,a),
\end{equation}
and mapping $a_{\min}:[0,L] \times \mathbb{M}^3_{\skw} \to \R$ that satisfies
\begin{equation} \label{defa}
 Q_0(x_1,A)=Q(x_1,A,a(x_1, A)).
\end{equation}
It is easy to see that  $Q_0$ satisfies the following property.
\begin{itemize}
 \item[($Q'_{0}1$)] For  almost all $x_1\in [0,L]$ the map $Q_0(x_1,\cdot,\cdot)$ is a quadratic form and
      satisfies
    \begin{equation*}
      \eta_1 \min\{\mu_1,\mu_2,\bar{C}_{\omega}\} |A|^2 \leq Q_0 (x_1,A)\leq \eta_2 \max\{\mu_1,\mu_2\}|A|^2 \qquad \text{for all $A \in \mathbb{M}^3_{\skw}$,}
    \end{equation*}
where $\bar{C}_{\omega}$ is defined in (\ref{defC}).
\end{itemize}
The mapping $a_{\min}$ is well defined, linear in $A$ and for some $C_a>0$ we have
$$ |a_{\min}(x_1,A)| \leq C_a |A|, \textrm{ for a.e. } x_1 \in [0,L]    .$$
%We define $\mathcal{L}: [0,L] \times \R^4 \to \R^4$, such that for a.e. $x_1 \in [0,L]$, $\mathcal{L}(x_1,\cdot)$ is positive definite and
%$$Q(x_1,A,a)=\mathcal{L}(x_1,(\axl A,a)) \cdot (\axl A,a). $$
%Analogously we define $\mathcal{L}^0: [0,L] \times \R^3 \to \R^3$, i.e., such that
%$$Q(x_1,A)=\mathcal{L}^0(x_1,\axl A) \cdot \axl A. $$
\subsection{Identification of $\Gamma$-limit} We will state and prove liminf and limsup inequality.
\begin{theorem}\label{theorem1}
Let Assumption \ref{konacna} be valid.
Assume that the sequence of deformations $(y^h)_{h>0} \subset W^{1,2}(\Omega;\R^3)$ satisfy \eqref{uvjetnastationary}. Then every sequence $(h_n)_{n \in \N}$  monotonly decreasing to zero has its subsequence (still denoted by $(h_n)_{n \in \N}$)  such that the following is valid
\begin{enumerate}
\item there exists $R \in W^{1,2}([0,L];\SO 3)$ such that $\nabla_{h_n} y^{h_n} \to R$ strongly in $L^2$.
\item $$ \liminf_{n \to \infty} \tfrac{1}{h_n^2} \int_{\Omega} W^{h_n}(x,\nabla_{h_n} y^{h_n}) \geq \int_{[0,L]} Q_0(R^t R')dx_1. $$
\end{enumerate}
\end{theorem}
\begin{proof}
We take $(R^h)_{h>0} \subset C^{\infty}([0,L]; \R^{3 \times 3})$ such that $R^h(x_1) \in \SO 3$ for a.e. $x_1 \in[0,L]$ and $R^h$  satisfies
\eqref{odnos1} and \eqref{odnos2}.
From \eqref{odnos2} we conclude that on a subsequence $R^{h_n} \rightharpoonup R$ weakly in $W^{1,2}([0,L];\R^3)$ and thus also in $C([0,L];\R^3)$.
We write the following decomposition
\begin{equation} \label{dekomp}
 y^{h_n}=\tfrac{1}{|\omega|} \int_{\{x_1\} \times \omega} y^{h_n} +h_nx_2 R^{h_n} e_2+h_nx_3R^{h_n} e_3+h_n v^{h_n}.
\end{equation}
Using \eqref{normalizacija} it is easy to see that for a.e. $x_1 \in [0,L]$ we have
\begin{equation} \label{norma1}
\int_{\{x_1\} \times \omega} v^{h_n} dx=0.
\end{equation}
We prove that $\|\nabla_{h_n} v^{h_n}\|_{L^2}$ is bounded. It is easy to see
$$\nabla_{h_n} v^{h_n}=\tfrac{1}{h_n}(\nabla_{h_n} y^{h_n}-R^{h_n})-\left(p^{h_n}+x_2(R^{h_n})'e_2+x_3 (R^{h_n})'e_3|0|0 \right), $$
where
$$ p^{h_n}= \tfrac{1}{{h_n}|\omega|} \int_{\{x_1\} \times \omega} \left(\partial_1 y^{h_n}-R^{h_n} e_1\right). $$
Using \eqref{odnos1} we conclude that there exists $C>0$ such that
\begin{equation} \label{ocjenap}
\|p^{h_n}\|_{L^2} \leq C.
\end{equation}
Using \eqref{odnos1} and \eqref{odnos2} we conclude that there exists $C>0$ such that
\begin{equation} \label{norma2}
\|\nabla_{h_n} v^{h_n}\|_{L^2} \leq C.
\end{equation}
Using \eqref{norma1} and Poincare inequality  we conclude that for some $C>0$
\begin{equation}\label{ocjenav}
\|v^{h_n}\|_{L^2} \leq Ch_n.
\end{equation}
Define the approximate strain
\begin{equation} \label{strain}
 G^{h} =\frac{(R^{h})^T \nabla_{h} y^{h} -I}{h}.
\end{equation}
From \eqref{odnos1} we conclude that $(G^h)_{h>0}$ is bounded in $L^2$.
It can be easily seen that
\begin{equation}\label{racun1}
G^{h_n}=\iota \left((A+A^{h_n}) d_{\omega}\right)+((R^{h_n})^t p^{h_n}|0|0)+(R^{h_n})^t \nabla_{h_n} v^{h_n},
\end{equation}
where $A=R^tR'$, $A^{h_n}=(R^{h_n})^t (R^{h_n})'-R^tR'$. Let $p \in L^2([0,L];\R^3)$ such that $p^{h_n} \rightharpoonup p$ weakly in $L^2$ (on a subsequence).
Take $(r_n)_{n \in \N} \subset C^1([0,L];\R^3)$ such that
\begin{equation}
r_n \to R^t p, \quad h_n (r_n)' \to 0, \textrm{ strongly in } L^2.
\end{equation}
Define
\begin{eqnarray*}
\tilde{p}^{h_n}&=&\int_0^{x_1} \left((R^{h_n})^tp^{h_n}-R^tp\right),\\
\tilde{v}^{h_n}&=&(R^{h_n})^t v^{h_n}+(h_nx_2 r_{n,2}+h_nx_3 r_{n,3},0,0)^t+\tilde{p}^{h_n}, \\
o^{h_n} &=& (R^{h_n})^t \nabla_{h_n} v^{h_n}-\nabla_{h_n} ((R^{h_n})^t v^{h_n})-(h_n x_2 (r_{n,2})'+h_n x_3 (r_{n,3})') e_1 \otimes e_1\\ & &+\sum_{i=2,3}\left(r_{n,i}-(R^t p)_i\right)e_i \otimes e_1.  \\
\tilde{A}^{h_n} &=& \int_0^{x_1} A^{h_n}.
\end{eqnarray*}
Notice that
\begin{equation*}
(R^h)^t \nabla_{h_n} v^{h_n}-\nabla_{h_n} ((R^{h_n})^t v^{h_n})=-((R^{h_n})'v^{h_n}|0|0) \to 0, \textrm{ strongly in } L^2.
\end{equation*}
This follows from \eqref{odnos2} and \eqref{ocjenav}, since we have that $\|h_n (R^{h_n})'\|_{L^\infty} \to 0$, by the Sobolev embedding. From this it follows that
\begin{equation} \label{ocjenao}
o^{h_n} \to 0 \textrm{ strongly in } L^2.
\end{equation}
It also easily follows that
\begin{equation}\label{ocjenaav}
\tilde{A}^{h_n} \to 0, \ \tilde{v}^{h_n} \to 0, \textrm{ strongly in } L^2, \quad \|\nabla_{h_n} \tilde{v}^{h_n}\|_{L^2} \leq C,
\end{equation}
for some $C>0$.
Observe that
\begin{equation}
\sym G^{h_n}= \sym \iota(Ad_{\omega})+(R^tp)_1 e_1 \otimes e_1+\sym \iota \left( (\tilde{A}^{h_n})'d_{\omega} \right)+\sym \nabla_h \tilde{v}^{h_n}+\sym o^{h_n}.
\end{equation}
Now using Lemma \ref{prvavazna} and Lemma \ref{drugavazna} we take a subsequence,  $(h_{n(k)})_{k \in \N}$ such that there exist sequences $(\bar{A}_k)_{k \in \N} \subset W^{1,2}([0,L]; M^3_{\skw})$,
$\bar{v}_k \subset W^{1,2}(\Omega; \R^3)$ which satisfy
\begin{enumerate}[(i)]
\item
$ \lim_{k \to \infty} |\{\bar{v}_k \neq   \tilde{v}^{h_{n(k)}} \textrm{ or }  \nabla \bar{v}_k \neq  \nabla \tilde{v}^{h_{n(k)}} \}|=0,$ \\ $ \lim_{k \to \infty} |\{\bar{A}_k \neq   \tilde{A}^{h_{n(k)}} \textrm{ or }   \bar{A}_k' \neq   (\tilde{A}^{h_{n(k)}})' \}|=0.$
\item $|\bar{A}_k'|^2$ and  $(|\nabla_{h_{n(k)}} \bar{v}_k|^2)_{k \in \N}$ are equi-integrable.
\end{enumerate}
It can be easily seen that $\bar{A}_k \to 0$, $\bar{v}_k \to 0$, strongly in $L^2$ (i.e. weakly in $W^{1,2}$).  By using Lemma \ref{peter1} we obtain a sequence $(\psi_k)_{k \in \N} \subset W^{1,2}(\Omega; \R^3)$ such that
\begin{eqnarray}\label{defpsi}
&& \sym \nabla_{h_{n(k)}} \psi_k = \sym \iota ((A_k)'d_{\omega})+\sym \nabla_{h_{n(k)}} \bar{v}^{h_{n(k)}}, \\ \nonumber &&(\psi_{k,1},h_{n(k)} \psi_{k,2}, h_{n(k)} \psi_{k,3}) \to 0, \ \int_{\omega} x_3 \psi_{k,2} \to 0 \textrm{ strongly in }  L^2.
\end{eqnarray}
The sequence $(| \sym \nabla_{h_{n(k)}} \psi_k|^2)_{k \in \N}$ is equi-integrable.
We define the sets
$$ C^h=\{ x \in \Omega: |G^h| \leq \tfrac{1}{\sqrt{h}} \}. $$
From the boundedness of the sequnce $(G^h)_{h>0}$ we conclude that $|\Omega \backslash C^h| \to 0$ as $h \to 0$.
Using frame indifference property we have that $W^h(x, \nabla_h y^h)=W^h(x,I+hG^h)$. From \eqref{eq:94}, by integrating, we conclude that
\begin{equation}\label{donjafinal}
\limsup_{h \to 0} \left|\tfrac{1}{h^2}\int_{\Omega} W^h (\cdot,I+h\chi_{C^h} G^h)-\int_{\Omega} Q^h(\cdot,\chi_{C^h} G^h) \right| \leq  r(\sqrt{h}) \int_{\Omega}|\chi_{C^h} G^h|^2\to 0.
\end{equation}
Finally we conclude, using the equi-integrability of  $(| \sym \nabla_{h_{n(k)}} \psi_k|^2)_{k \in \N}$, (Q1), the definition of $K$ and \eqref{donjafinal}:
\begin{eqnarray*}
& &\liminf_{k \to \infty} \tfrac{1}{h_{n(k)} ^2} \int_{\Omega} W^{h_{n(k)}}(x,\nabla_{h_{n(k)}} y^{h_{n(k)}}) \\ & &  \geq  \liminf_{k \to \infty} \tfrac{1}{h_{n(k)} ^2} \int_{\Omega}\chi_{C^{h_{n(k)}}} W^{h_{n(k)}}(x,\nabla_{h_{n(k)}} y^{h_{n(k)}})\\& &= \liminf_{k \to \infty} \int_{\Omega} Q^{h_{n(k)}}(x,\chi_{C^{h_{n(k)}}} G^h) \\ & &
= \liminf_{k \to \infty} \int_{\Omega} Q^{h_{n(k)}} \left(x,\chi_{C^{h_{n(k)}}} \left(  \iota(Ad_{\omega})+(R^tp)_1 e_1 \otimes e_1+ \iota ( (\tilde{A}^{h_n})'d_{\omega} )+ \nabla_h \tilde{v}^{h_n}       \right) \right) \\& &= \liminf_{k \to \infty} \int_{\Omega} Q^{h_{n(k)}} \left(x, \iota(Ad_{\omega})+(R^tp)_1e_1 \otimes e_1+\nabla_{h_{n(k)}} \psi^{h_{n(k)}} \right) \\ & & \geq K(m(A,(R^tp)_1,[0,L]) \\ & & \geq
\int_{[0,L]} Q_0\left(R^t R'(x_1)\right)dx_1.
\end{eqnarray*}

\end{proof}
The next theorem gives the construction of the recovery sequence.
\begin{theorem}\label{theorem2}
Let Assumption \ref{konacna} be valid. Then for every $R \in W^{1,2}([0,L];\SO 3)$ and every sequence $(h_n)_{n \in \N}$ monotonly decreasing to $0$ there exists a subsequence, still denoted by $(h_n)_{n \in \N}$, such that
\begin{enumerate}
\item there exists $(y_n)_{n \in \N} \subset W^{1,2}(\Omega;\R^3)$ such that $y_n \to \int_0^{x_1} Re_1$ strongly in $W^{1,2}$, $\nabla_{h_n} y_n \to R$ strongly in $L^2$.
\item $\lim_{n \to \infty} \frac{1}{h_n^2} \int_{\Omega} W^{h_n} (x, \nabla_{h_n} y_n)=\int_{[0,L]} Q_0\left(R^t R'(x_1)\right) \ud x_1$.
\end{enumerate}
\end{theorem}
\begin{proof}
It is easy to see that smooth rotations are dense in $W^{1,2}([0,L];\SO 3)$. This can be seen by approximating with smooth maps taking values in $\mathbb{M}^3$ and then projecting on $\SO 3$ (by Sobolev embedding weak $W^{1,2}$ implies strong convergence in $L^{\infty}$ and we can project from tubular neighbourhood of $\SO 3$).
Without loss of generality we can assume that $R \in C^2([0,L]; \SO 3)$, since in the general case we can use the diagonal procedure. Take $a \in C([0,L])$ and define $A \in C^1([0,L];\mathbb{M}^3_{\skw})$ as $A=R^t R'$.
Now we take $m=m(R^t R',a)$ and the sequence $(\vartheta_n(m))_{n \in \N} \subset W^{1,2}(\Omega;\R^3)$ which satisfies (a) and (b) of the Assumption \ref{konacna}. From Corollary \ref{pommmm} we have the boundedness and equi-integrability of $(|\sym \nabla_{h_n} \vartheta_n(m)|^2)_{n \in \N}$ (see \eqref{bound}).Using Corollary \ref{cormaroje},  we obtain a sequence $(A_n)_{n \in \N} \subset W^{1,2}([0,L];\mathbb{M}^3_{\skw})$, $(v_n)_{n \in \N} \subset W^{1,2}(\Omega;\R^3)$ such that $A_n\to 0$, $v_n \to 0$ strongly in $L^2$ and
$$ \| \sym \nabla_{h_n} \vartheta_n(m)-\sym \iota(A'_n d_{\omega})-\sym \nabla_{h_n} v_n\|_{L^2} \to 0. $$
Moreover we have that
 \begin{equation}
\sup_{n \in \N} \|A_n\|_{W^{1,2}} + \sup_{n \in \N}\left(\|v_n\|_{L^2}+\|\nabla_{h_n} v_n\|_{L^2}\right) <\infty.
 \end{equation}
Choose a subsequence $(h_{n(k)})_{k \in \N}$  such that $k h_{n(k)} \to 0$. Using
Lemma \ref{prvavazna} and Lemma \ref{drugavazna} we conclude that there exist sequences $(\tilde{A}_k)_{k \in \N} \subset W^{1,\infty}([0,L];\mathbb{M}^3_{\skw})$ and $(\tilde{v}_k)_{k \in \N} \subset W^{1,\infty}(\Omega;\R^3)$ such that for some $C>0$ we have (on a further subsequence; not relabeled)
\begin{enumerate}
\item $|A_k'| \leq Ck$,  for a.e. $x_1 \in [0,L]$, $|\nabla_{h_{n(k)}} \tilde{v}_{k} | \leq Ck$ for a.e. $x \in \Omega$.
    \item $ \lim_{k \to \infty} |\{\tilde{A}_k \neq A_{n(k)} \textrm{ or } \tilde{A}'_k \neq A'_{n(k)}\}|=0$ \\
        $ \lim_{k \to \infty} |\{\tilde{v}_k \neq v_{n(k)} \textrm{ or } \nabla \tilde{v}_k \neq \nabla v_{n(k)} \}|=0$
\item the sequences $(|\tilde{A}_k'|^2)_{k \in \N}$, $(|\nabla_{h_{n(k)}} \tilde{v}_k|^2)_{k \in \N}$ are equi-integrable.
\end{enumerate}
It is easy to argument that $\tilde{A}_k \to 0$, $\tilde{v}_k \to 0$ strongly in $L^2$ (i.e. weakly in $W^{1,2}$).
We define the sequence $(R_k)_{k \in \N} \subset C^1([0,L];\mathbb{M}^3)$ as the solutions of the following Cauchy problem
\begin{equation}\label{ODEsystem}
\left\{ \begin{array}{rcl}R_k'&=& R_k(A+\tilde{A}'_k), \\ R_k(0) &=& R(0).   \end{array}\right.
\end{equation}
Since the right hand side of the first equation in \eqref{ODEsystem}  is Lipschitz function this system has unique solution. Moreover, since it is tangential to $\SO 3$  it can be easily argumented that we have $R_k(x_1) \in SO(3)$ for every $x_1 \in [0,L]$ (this can be done e.g. by approximating $A_k$ with smooth fields and then using the standard theorem for the solutions of ODE system whose right hand side is tangential to some smooth manifold). Notice also that $R_k \rightharpoonup R$ weakly in $W^{1,2}$ and thus, by Sobolev embedding strongly in $L^\infty$.
Define for every $k \in \N$;  $\bar{v}_k=\tilde{v}_k-\int_{\Omega} \tilde{v}_k$ to accomplish $\|\bar{v}_k\|_{W^{1,\infty}} \leq Ck$, which follows by Poincare inequality.
 Using the equi-integrability property it is easy to see that
\begin{equation} \label{ekvi}
\| \sym \nabla_{h_{n(k)}} \vartheta_{n(k)}-\sym \iota(\tilde{A}'_k d_{\omega})-\sym \nabla_{h_{n(k)}} \bar{v}_k\|_{L^2} \to 0.
\end{equation}
Define the recovery sequence with the formulae
 \begin{eqnarray*}
 y_k&=&\int_0^{x_1} R_k e_1 +h_{n(k)} x_2 R_k e_2+h_{n(k)} x_3 R_k e_3+h_{n(k)} R \bar{v}_k \\
 & &-h_{n(k)}^2\left(x_2(R^t R' \bar{v}_k)_2-x_3(R^t R' \bar{v}_k)_3\right)Re_1+h_{n(k)}\int_0^{x_1}(a-(R^t R'\bar{v}_k)_1)Re_1.
 \end{eqnarray*}
Define also
$$G_k=\frac{R_k^t \nabla_{h_{n(k)}} y_k-I}{h_{n(k)}}. $$
It is easy to see that
\begin{enumerate}
\item $\|y_k-\int_0^{x_1} Re_1 \|_{L^{\infty}} \to 0$, $\| \nabla_{h_{n(k)}} y_k -R_k\|_{L^{\infty}} \to 0$,
\item $\|h_{n(k)} G_k\|_{L^{\infty}} \to 0$,
\item $\left\| \sym G_k-a e_1 \otimes e_1-\sym \iota\left((A+\tilde{A}_k') d_{\omega} \right)-\sym \nabla_{h_{n(k)}} \bar{v}_k\right\|_{L^2} \to 0$.
\end{enumerate}
From \eqref{ekvi} we conclude that
\begin{equation} \label{zadnjaocjena}
\left\| \sym G_k-a e_1 \otimes e_1-\sym \iota\left(A d_{\omega} \right)-\sym \nabla_{h_{n(k)}} \vartheta_{n(k)} \right\|_{L^2} \to 0.
\end{equation}
Notice that from the property (W1) of Definition \ref{def:materialclass} we have that $W^{h_{n(k)}}(x,\nabla_{h_{n(k)}} y_k)= W^{h_{n(k)}}(x,I+h_{n(k)}G_k)$, for a.e. $x \in \Omega$.
Using property (iii) of Definition \ref{def:composite} as well as property (b) of $G_k$ we conclude that
$$\left| \tfrac{1}{h^2_{n(k)}} \int_{\Omega} W^{h_{n(k)}}(x,\nabla_{h_{n(k)}}y_k)- \int_{\Omega} Q^{h_{n(k)}}(x, G_k)\right| \to 0.$$
Using \eqref{ekvi} we conclude that
$$ \left|\lim_{k \to \infty} \int_{\Omega} Q^{h_{n(k)}}(x, G_k)-\int_{[0,L]} Q(x_1,A,a)\right| \to 0. $$
The claim now follows by diagonilizing procedure and approximating $a_{\min} (\cdot,A(\cdot))\in L^{\infty} ([0,L])$, defined in \eqref{defa}, with continuous maps in $L^2$ norm.
\end{proof}
\section{Appendix}

We give two technical lemmas.
\begin{lemma}\label{prvavazna}
Let $\Omega \subset \R^N$ be a Lipschitz set and $p>1$. Let $(w_n)_{n \in \N}$ be  a bounded sequence in $W^{1,p}(\Omega;\R^m)$. There exists a  subsequence $(w_{n(k)})_{k \in \N}$  such that  for every $k \in \N$ there exists $z_{k} \in W^{1,p}(\Omega;\R^m)$  which satisfies
\begin{enumerate}[(i)]
\item $ |\nabla z_k| \leq C(N)k, \textrm{ for a.e. } x \in \Omega.$
%5Here
%$$ R_n=\{x \in \Omega: M(\nabla w_{n(k)}) \geq n\}, $$
% $$M(\nabla w_{n(k)})(x)= \sup_{r>0}\tfrac{1}{B(x,r)} \int_{B(x,r)} |\nabla \tilde{w}_{n(k)}(x)|\ud x, $$
%and $\tilde{w}_{n(k)} \in W^{1,p} (\R^N;\R^M)$ is an extension of $w_{n(k)}$ given by some fixed extension %operator.
\item
$ \lim_{ \to \infty} |\Omega  \cap \{ z_{k} \neq w_{n(k)} \textrm{ or } \nabla z_{k} \neq  \nabla w_{n(k)} \}|=0. $
\item  $(|\nabla z_{k}|^p)_{k \in \N}$ is equi-integrable.
\end{enumerate}
%If necessary, one can additionally assume that $n(k) \geq \tilde{k}(n)$ for every $n \in \N$ and some increasing %function $\tilde{k}: \N \to \N$.
\end{lemma}
\begin{proof}
The proof is implicitly contained in the proof of Lemma 1.2. (decomposition lemma) in \cite{FoMuPe98}. We shall skip it here.
\end{proof}
\begin{lemma}\label{drugavazna}
Let $\omega \subset \R^2$ be a set with Lipschitz boundary, let $\Omega=[0,L] \times \omega$ and  let $p>1$.
Let $(w^h)_{h>0}$ be a sequence
bounded in $W^{1,p}(\Omega;\R^m)$ and let us additionally assume that the sequence
$(\|\nabla_h w^h\|_{L^p})_{h>0}$ is bounded.
Then for every sequence $(w^{h_n})_{n \in \N}$ there exists a  subsequence $(w^{h_{n(k)}})_{k \in \N}$  such that  for every $k \in \N$ there exists $z_{k} \in W^{1,p}(\Omega;\R^m)$  which satisfies
\begin{enumerate}[(i)]
\item $ |\nabla_{h_{n(k)}} z_k| \leq C(N)k, \textrm{ for a.e. } x \in \Omega.$
\item
$ \lim_{k \to \infty} |\Omega  \cap \{ z_{k} \neq w^{h_{n(k)}} \textrm{ or } \nabla z_{k} \neq  \nabla w^{h_{n(k)}} \}|=0. $
\item  $(|\nabla_{h_{n(k)}} z_{k}|^p)_{k \in \N}$ is equi-integrable.
\end{enumerate}
%If necessary, one can additionally assume that $n(k) \geq \tilde{k}(n)$ for every $n \in \N$ and some increasing %function $\tilde{k}: \N \to \N$.
\end{lemma}
\begin{proof}
In \cite{BraidesZeppieri07} the authors provide a general proof for the function space $W^{1,p}(\omega_\alpha\times \omega_\beta;\mathbb{R}^m)$ where $\omega_\alpha \subset \mathbb{R}^{n}$ and $\omega_\beta \subset \R^{l}$ and $\left\{n,m,l\right\}$ are arbitrary space dimensions. For completeness  we give the proof for our case.

By de la Vall\'{e}e Poussin's Criterion a sequence $(\zeta_k)_{k \in \N} \subset L^{1}(\Omega;\R^m)$ is equi-integrabile if and only if there exists a positive Borel function   $\varphi:\left[0,\infty\right) \to \left[0,\infty\right]$ such that
$$\lim_{t \to +\infty}{\frac{\varphi(t)}{t}}=+\infty  \quad \mbox{and} \quad \sup_{k} \int_{\Omega} \varphi(\vert \zeta_k\vert) < + \infty.$$

By translation and dilatation we can assume without loss of generality that $\omega \subset Q^2$, where $Q^2=(0,1)^2$.
Let $(w^h)_{h>0}$ be a given bounded sequence in $W^{1,p}(\Omega,\R^3)$ such that $\Vert (\nabla_h w^h)\Vert_{L^p}$ is also bounded. By using standard extension techniques  we extend the definition of $w^h$ to $W^{1,p}((0,L)\times Q^2;\R^3)$ (the extension is done for every fixed $x_1 \in (0,L)$), while keeping the boundness properties.
We break the proof in several steps.
\begin{enumerate}
\item[1.] Define the functions $\hat{w}^{h_n}(x):=w^{h_n}\left(x_1,\frac{x'}{h_n} \right)$ on  a strip $(0,L)\times(0,h_n)^2$. Then $\hat{w}^{h_n}$ is in $ W^{1,p}((0,L)\times(0,h_n)^2;\R^3)$ and from the boundness of $w^{h_n}$ and $\nabla_{h_n}w^{h_n}$, by rescaling the integrals on the new domain we obtain that there is a constant $C>0$ such that
\begin{equation}\label{scalest}
 \frac{1}{h_n^2} \int_{(0,L)\times(0,h_n)^2} \vert \hat{w}^{h_n}\vert^p \ud x +  \frac{1}{h_n^2} \int_{(0,L)\times(0,h_n)^2} \left(\vert \partial_1{\hat{w}^{h_n}}\vert^p +\vert \nabla'\hat{w}^{h_n}\vert^p\right)  \ud x \leq C.
\end{equation}

\item[2.] Next, define $\tilde{w}^{h_n}$ on $(0,L)\times(-h_n,h_n)^2$ by reflecting the functions $\hat{w}^{h_n}$ with respect to the $x_2$ and $x_3$ variable $\tilde{w}^{h_n}(x)=\hat{w}^{h_n}(x_1,\vert x_2\vert,\vert x_3\vert)$. We define the functions $\bar{w}^{h_n}(x)(x)=\tilde{w}^{h_n}(x_1,x'-(2i h_k ,2j h_k )), i,j\in \mathbb{Z}$ on the $(0,L)\times \R^2$ by periodically extending $\tilde{w}^{h_n}$. From the construction of $\tilde{w}^{h_n}$ it easy to see that $\bar{w}^{h_n} \in W^{1,p}_{loc}((0,L)\times \R^2,\R^3)$. Now since $(0,L) \times Q^2$ is contained in $(\lfloor \frac{1}{2h_n}\rfloor+2)^2$ cubes and since $\tilde{w}^{h_n}$ is symmetric with respect to $x_2$ and $x_3$ axes we derive that for $n$ large enough:
\begin{eqnarray*}
\int_{(0,L) \times Q^2} \vert \bar{w}^{h_n}\vert^p \ud x &\leq& 4 \left(2+\left\lfloor \frac{1}{2h_n}\right\rfloor\right)^2 \int_{(0,L)\times(0,h_n)^2} \vert \hat{w}^{h_n} \vert^p \ud x\\  &\leq& \frac{4}{h_n^2} \int_{(0,L)\times(0,h_n)^2} \vert \hat{w}^{h_n} \vert^p \ud x.
\end{eqnarray*}
Thus,  from (\ref{scalest}) we deduce that  $\tilde{w}^{h_n}$ is bounded with respect to $n$. Using the same arguments  the gradients $\nabla\bar{w}^{h_n}$ are also bounded with respect to $n$.

\item[3.] Since the sequences $(\bar{w}^{h_n})_{n \in \N}$ satisfy the assumptions of the lemma (\ref{prvavazna}), there is a sequence $(v_k)_{k \in \N} \subset W^{1,p}((0,L)\times Q^2)$ such that $\vert \nabla v_k \vert< C(N) k $ \emph{a.e.} on $(0,L)\times Q^2$ and
$$ \lim_{k \to \infty} \left|(0,L)\times Q^2  \cap \{ v_{k} \neq w^{h_{n(k)}} \textrm{ or } \nabla v_{k} \neq  \nabla w^{h_{n(k)}} \} \right|=0 $$
and $(\vert \nabla v_k \vert^p )$ is equi-integrable on $(0,L)\times Q^2$. By de la Vall\'{e}e Poussin's criterion  then there is a positive Borel function $\varphi:\left[0,\infty\right) \to \left[0,\infty\right]$ such that
$$\lim_{t \to +\infty}{\frac{\varphi(t)}{t}}=+\infty  \quad \mbox{and} \quad \sup_{k} \int_{(0,L)\times Q^2} \varphi(\vert \nabla v_k\vert^p) < + \infty.$$
 We denote with
 \begin{eqnarray*}
 M_k &=&\int_{(0,L)\times Q^2} \varphi(\vert \nabla v_k\vert^p) \\
 m_k &=&\left|(0,L)\times Q^2  \cap \{ v_{k} \neq w^{h_{n(k)}} \textrm{ or } \nabla v_{k} \neq  \nabla w^{h_{n(k)}} \}\right|
 \end{eqnarray*}
and (by Lemma \ref{prvavazna})  $\sup_{k} M_k<\infty$ and $\lim_{k \to \infty} m_k =0$.
\item[4.]
It is easy to argument that for $k$ large enough there exists a part of the domain $S_j^{h_{n(k)}} \subset (0,L) \times Q^2$ of the form $S_j^{h_{n(k)}}=(0,L) \times (jh_{n(k)},(j+1)h_{n(k)})^2$ such that
\begin{eqnarray*}
\int_{S_j^{h_{n(k)}}} \varphi(\vert \nabla v_k\vert^p) &\leq&  3 h_{n(k)}^2 M_k, \\
\left|S_j^{h_{n(k)}} \cap \{ v_{k} \neq w^{h_{n(k)}} \textrm{ or } \nabla z_{k} \neq  \nabla w^{h_{n(k)}} \}\right| &\leq& 3 h_{n(k)}^2 m_k.
\end{eqnarray*}

\item[5.] Finally, we  define the functions $\tilde{z}_k=v_k|_{S_j^{h_{n(k)}}}$ and the functions $z_k \in W^{1,p}((0,L) \times Q^2;\R^3)$ by translation, dilatation in $x_2,x_3$ variable and possible reflection of the functions $\tilde{z}_k$.
\end{enumerate}
\end{proof}

{\bf Acknowledgement.}
The authors were informed about the decomposition in Lemma \ref{lmgriso} at the conference "Third workshop on thin structures" in Naples, September, 2013. It was there announced as part of the work (and more general theorem) of J. Casado-Diaz, M. Luna-Laynez and F.J. Suarez-Grau.
Here we used that in the rod case this decomposition is a consequence of Griso's decomposition.
We are gratefull to M. Luna-Laynez on a personal communication.

\bibliographystyle{alpha}

\end{document}